%% file: sbvSystemGNarxiv.tex
\documentclass[a4paper,reqno]{amsart}

\usepackage[active]{srcltx}
\usepackage[all]{xy}
\usepackage{subfig}
\usepackage{amssymb,latexsym}
\usepackage{mathrsfs}
\usepackage{graphics}
\usepackage{pst-pdf}
\usepackage{latexsym}
\usepackage{psfrag}
\usepackage{verbatim}
\usepackage{graphicx}
\usepackage{xcolor}
\usepackage{pifont,marvosym}
\usepackage{psfrag}

\theoremstyle{plain}
\newtheorem{lemma}{Lemma}[section]
\newtheorem{proposition}[lemma]{ Proposition}
\newtheorem{theorem}[lemma]{ Theorem}
\newtheorem{theorem*}{ Theorem}
\newtheorem{corollary}[lemma]{ Corollary}

\theoremstyle{definition}
\newtheorem{definition}[lemma]{Definition}

\theoremstyle{remark}
\newtheorem{remark}[lemma]{Remark}
\newtheorem{example}[lemma]{Example}
\newtheorem{notation}[lemma]{Notation}

\numberwithin{equation}{section}

\DeclareMathOperator{\Graph}{Graph}

\newcommand{\paragr}[1]{\paragraph{\texttt {#1}}}
\newcommand{\restr}[1]{_{\llcorner{#1}}}
\newcommand{\wlim}{\mathrm{weak}^{*}\!\!\!\!\mathrm{-}\!\!\lim}

\newcommand{\OO}{\mathcal{O}(1)}
\newcommand{\R}{\mathbb{R}}

\newcommand{\N}{\mathbb{N}}

\newcommand{\TV}{\mathrm {Tot.Var.}}
\newcommand{\BV}{\mathrm {BV}}
\newcommand{\SBV}{\mathrm {SBV}}

\newcommand{\rc}{\text{\rm c}}
\newcommand{\loc}{\text{\rm loc}}

\newcommand{\Ll}{\mathcal{L}}

\newcommand{\jump}{\mathrm{jump}}
\newcommand{\inn}{\mathrm{in}}
\newcommand{\out}{\mathrm{out}}
\newcommand{\cont}{\mathrm{cont}}

\newcommand{\Dom}{\mathrm{Dom}}

\newcommand{\dmns}{\mathrm{N}}

\DeclareMathOperator{\Dif}{D\!\!}
\DeclareMathOperator{\Difc}{D^{c}}
\DeclareMathOperator{\Difj}{D^{j}}
\DeclareMathOperator{\Difxa}{D_{\textit{x}}^{a}}
\DeclareMathOperator{\Difxc}{D_{\textit{x}}^{c}}
\DeclareMathOperator{\Diftc}{D_{\textit t}^{c}}
\DeclareMathOperator{\Difxj}{D_{\textit x}^{j}}

\title[SBV Regularity for  Genuinely Nonlinear 1D Hyperbolic Systems of Conservation Laws]{SBV Regularity for  Genuinely Nonlinear, Strictly Hyperbolic Systems of Conservation Laws in one space dimension}

\author{Stefano Bianchini}
\address{SISSA, via Bonomea 265, I-34014 Trieste (ITALY)}
\email{bianchin@sissa.it}

\author{Laura Caravenna}
\address{Centro De Giorgi, Scuola Normale Superiore, Piazza dei Cavalieri, 3
I-56100 Pisa (ITALY)}
\email{laura.caravenna@sns.it}

\date{October, 2010}

\thanks{The authors wish to thank Giovanni Alberti for the useful discussions about this paper. 
}

\keywords{Conservation laws; Oleinik's inequality; front-tracking algorithm; regularity estimates.}
\begin{document}
\subjclass[2000]{Primary 35L65; Secondary 35B05, 35D10}

\begin{abstract}
We prove that if $t \mapsto u(t) \in \BV(\R)$ is the entropy solution to a $N \times N$ strictly hyperbolic system of conservation laws with genuinely nonlinear characteristic fields
\[
u_t + f(u)_x = 0,
\]
then up to a countable set of times $\{t_n\}_{n \in \N}$ the function $u(t)$ is in SBV, i.e. its distributional derivative $u_x$ is a measure with no Cantorian part.

The proof is based on the decomposition of $u_x(t)$ into waves belonging to the characteristic families
\[
u(t) = \sum_{i=1}^N v_i(t) \tilde r_i(t), \quad v_i(t) \in \mathcal M(\R), \ \tilde r_i(t) \in \R^N,
\]
and the balance of the continuous/jump part of the measures $v_i$ in regions bounded by characteristics. To this aim, a new interaction measure $\mu_{i,\jump}$ is introduced, controlling the creation of atoms in the measure $v_i(t)$.

The main argument of the proof is that for all $t$ where the Cantorian part of $v_i$ is not $0$, either the Glimm functional has a downward jump, or there is a cancellation of waves or the measure $\mu_{i,\jump}$ is positive.
\end{abstract}

\maketitle

\tableofcontents

\section{Introduction}

In this paper we consider the entropy solution to the hyperbolic system in one space dimension
\begin{equation}
\label{E:consLaw}
\begin{cases}u_t + f(u)_x = 0,
\\
u(t=0)=\bar u
\end{cases}
 \qquad 
 \begin{array}{c}
 u :\R^+\times\R\to\Omega\subset\R^\dmns,
 \quad f\in C^{2}(\Omega,\R^{\dmns})
 \vspace{.2cm} \\
\bar u \in\BV(\R; \Omega),\quad | \bar u|_{\BV}\ll 1.
\end{array}
\end{equation}
We assume that each characteristic field is either genuinely nonlinear or linearly degenerate, and in what follows we will refer to the unique solution constructed by vanishing viscosity or wave-front tracking, see~\cite{Bressan,daf:book}.

While linearly degenerate families do not gain any regularity during the time evolution, the genuinely nonlinear families show a regularizing effect due to the non linearity of the eigenvalue. The most famous one is probably the decay of positive waves, which in the scalar case $N=1$ takes the form
\begin{equation}
\label{E:decaypos1}
u_x^+ \leq \frac{1}{k t},
\end{equation}
where $k$ is the genuinely nonlinearity constant,
\[
k = \inf_u f''(u) > 0.
\]
For a strictly hyperbolic system of conservation laws, even if the $i$-th family satisfies the genuinely nonlinearity condition
\[
D\lambda_i(u) r_i(u) \geq k > 0,
\]
where $\{\lambda_i\}_{i=1}^N$ and $\{r_i\}_{i=1}^N$ are the eigenvalues and eigenvectors of $A(u) := Df(u)$ with a suitable orientation, then it may happen that new positive waves are created at a later time. In this case, the estimate~\eqref{E:decaypos1} takes the form
\begin{equation}
\label{E:decaypos3}
v^+_i(t)(B) \leq C \bigg\{ \frac{\mathcal L^1(B)}{t-s} + Q(s) - Q(t) \bigg\}, \quad 0 \leq s < t,
\end{equation}
where $B$ is a Borel set, $Q$ is the Glimm interaction potential and the constant $C$ depends on $k$. The wave measures $v_i(t)$ are defined precisely in Section \ref{Ss:decowave}, and roughly speaking they are the part of $u_x(t)$ which has direction close to $r_i(u(t))$ and travels with a speed close to $\lambda_i(u(t))$.

\noindent Since $Q(t)$ is a decreasing function, an elementary argument yields that $v^+_i$ is absolutely continuous up to countably many times: in fact, if $B$ is $\mathcal L^1$-negligible and $v^+_i(t)(B) > 0$, then by letting $s \nearrow t$ we obtain
\[
Q(t) - \lim_{s \nearrow t} Q(s) \geq \frac{v^+_i(t)(B)}{C} > 0,
\]
so that $t \mapsto Q(t)$ has a jump downward. Being $Q$ decreasing, this can happen only countably many times.

A complementary estimate is the fact that also $v^-_i(t)$ has no Cantorian part. The first positive result has been given in~\cite{ambDel}, where it is shown that the solution $u(t)$ of a genuinely nonlinear scalar conservation law in one space dimension is SBV up to countably many times. In that paper, the authors consider the characteristic lines
\[
\dot x = f'(u(t,x)), \quad u(0,x) = y,
\]
and prove the following: every time a Cantorian part in $u_x(t)$ appears, then there is a set of positive measure $A$ such that all the characteristics starting from $y \in A$ are defined in the interval $[0,t]$ but cannot be prolonged more than $t$. By the $\sigma$-finitness of $\mathcal L^1$, one can apply the same observation used to prove that the positive part of $u_x(t)$ is abolutely continuous up to countably many times, and deduce that up to countably many times the solution $u(t)$ is SBV.

The use of the measure of the set $A(t)$ of initial points for characteristics which can be prolonged up to time $t$ has been applied to obtain extension of the above result: in~\cite{Roger} the SBV estimate is used for scalar balance laws, later extended to Temple systems in~\cite{anc} and in~\cite{BDR} to the case of Hamilton-Jacobi equation in several space dimension with uniformly convex Hamiltonian.
In the context of the Riemann problem for genuinely non-linear systems, the thesis has moreover been proved in~\cite{dafermos}.

The case of genuinely nonlinear systems of conservation laws is more complicate by the fact that centered rarefaction waves are created at $t > 0$, and thus the characteristics are not unique in the future. Thus, in estimating the $\mathcal L^1$-measure of the initial points, one has to take into account also that interaction points can generate centered rarefaction waves, so that the estimate should be something like
\begin{align}
\label{E:rough_balance}
\mathcal L^1 &\Big\{ \text{initial points of characteristics arriving at $t$ but not prolongable} \Big\} \crcr
&\geq \Big\{ \text{measure of the Cantorian part of $v_i(t)$} \Big\} - \Big\{ \text{amount of interaction in $[0,t]$} \Big\}.
\end{align}

In this paper we use a different approach. Let $\tilde \lambda_i(t,x)$ be the Rankine-Hugoniot speed if $u$ has a jump in the point $(t,x)$ of the $i$-th family or the $i$-eigenvalue of $A(u)$ in the remaining cases. We first prove that not only the characteristic waves $v_i(t)$ satisfy a balance equation of the form
\[
\partial_t v_i + \partial_x \big( \tilde \lambda_i v_i \big) = \mu_i,
\]
with $\mu_i$ a measure bounded by the interaction-cancellation of waves, but also its atomic part $v_{i,\jump}(t)$ satisfy
\[
\partial_t v_{i,\jump} + \partial_x \big( \tilde \lambda_i v_{i,\jump} \big) = \mu_{i,\jump},
\]
with $\mu_{i,\jump}$ bounded measure. This measure $\mu_{i,\jump}$ differs from the interaction-cancellation measure because it is not $0$ when an atomic part in $v_i$ is created, and it describes the natural behavior of solutions to genuinely nonlinear conservation laws: it is easy to create a shock because of the nonlinearity, but you can remove it only by means of cancellation or strong interactions.

The second step is to use the two above balance equations to study the balance of $v_i$, $v_{i,\jump}$ and $v_{i,\cont} = v_i - v_{i,\jump}$ in regions bounded by characteristics. The key estimate we obtain is that
\begin{equation}
\label{E:decaneg2}
v_{i,\cont}(t)(B) \geq - C \bigg\{ \frac{\mathcal L^1(B)}{\tau-t} + \mu^{ICJ} \big( \big\{[t,\tau] \times \R \big\} \big) \bigg\}, \quad 0 \leq t < \tau,
\end{equation}
where the measure $\mu^{ICJ}$ is the interaction-cancellation measure $\mu^{IC}$ plus morally the measure $\sum_{i=1}^N |\mu_{i,\jump}|$. This is the companion estimate of~\eqref{E:decaypos3}, and using the same argument of the positive part we conclude with the main result of this paper (Corollary~\ref{C:mainSBV}):

\begin{theorem}
\label{T:Main_SBV}
Let $u(t)$ be the entropy solution of the Cauchy problem
\[
\begin{cases}
u_t + f(u)_x = 0, 
\\
u(t=0)=\bar u
\end{cases}
\qquad u :\R^+\times\R\to \Omega\subset\R^\dmns,
\qquad f\in C^{2}(\Omega;\R^{\dmns})
\]
for a strictly hyperbolic system of conservation laws where each characteristic field is genuinely non-linear, with initial datum $\bar u$ small in $\BV(\R;\Omega)$.
Then $u(t)\in\SBV(\R;\Omega)$ out of at most countably many times.
\end{theorem}

\subsection{Structure of the paper}

The paper is organized as follows.

In Section~\ref{S:gene_prel_notion} the main notation and assumptions are introduced: strict hyperbolicity and characteristic families, and the decomposition into wave measures. A few fundamental results concerning the hyperbolic systems of conservation laws are recalled: Lax's solution to the Riemann problem (Theorem~\ref{T:sigmaLambda}) and Bressan's existence and uniqueness of a Lipschitz semigroup of solutions for small BV initial data (Theorem~\ref{T:semigroup_uniqueness}).

In Section~\ref{S:mainArg} we prove Theorem~\ref{T:Main_SBV} (Corollary~\ref{C:mainSBV} below). This result is a corollary of the fact that if the $\bar \imath$-th family is genuinely nonlinear, then the $\bar \imath$-th component $v_{\bar \imath}(t)$ of $u_x(t)$ is SBV up to countably many times (Corollary~\ref{C:viSBV}), and the latter is a consequence of the estimates~\eqref{E:decaypos3} and~\eqref{E:decaneg2} (Theorem~\ref{T:continuousEstimate}). In this section the notion of interaction-cancellation is recalled, and the SBV estimates are derived assuming~\eqref{E:decaneg2}, whose proof is postponed to Section~\ref{S:mainEstimate}.

Since the proof depends on uniform estimates for the wave-front tracking approximations, in Section~\ref{S:frontTrackingApp} we recall the basic properties of these approximated solutions. A key fact is the possibility to collect the jumps of the wave-front solution into two families: one is converging to the jump part of $u_x$, and the other to the continuous part of $u_x$. This is done by defining the maximal $(\varepsilon_0,\varepsilon_1)$-shocks (Definition~\ref{D:max_epsilon_shock}): $0 < \varepsilon_0 < \varepsilon_1$ are two treshold parameters, fixing the minimal size of the jump ($\varepsilon_0$) and the lower bound for the maximal size of the jump ($\varepsilon_1$). This definition has already been used in~\cite{Bressan} to study the structure of the semigroup solution $u$, which we recall in Theorem~\ref{T:pwconvergence}. From this result we obtain that the wave measure $v_i^\nu$, $v_{i,\jump}^\nu$ for the wave-front tracking approximation $u^\nu$ and the products $\tilde \lambda_i^\nu v^\nu_i$, $\tilde \lambda_i^\nu v^\nu_{i,\jump}$ converge weakly (Corollaries~\ref{C:corAprroximation},~\ref{C:hausdorfpendenze},~\ref{C:weakConvergences}).

In the last section, Section~\ref{S:mainEstimate}, we prove the decay estimate~\eqref{E:decaneg2}. First of all, we prove that the distributions
\begin{align*}
&\partial_t v^\nu_i + \partial_x \big( \tilde \lambda^\nu_i v^\nu_i \big),
&& \partial_t v^\nu_{i,\jump} + \partial_x \big( \tilde \lambda^\nu_i v^\nu_{i,\jump} \big)
\end{align*}
are uniformly bounded measures: we denote them respectively by $\mu_i^\nu$ and $\mu^\nu_{i,\jump}$ (Proposition~\ref{P:mu_estimate}). The latter measure is called the \emph{jump balance measure}. Since the bounds do not depend on the aprooximation parameter $\nu$, it is possible to pass to the limit and to obtain the balance equations for the wave measures $v_i$, $v_{i,\jump}$. The consequences of this fact however are not directly related with the SBV regularity, so we will address them in a forthcoming paper. Next, we study the balances of the measures $v^\nu_i$, $v^\nu_{i,\jump}$ in regions bounded by minimal characteristics (Lemma~\ref{L:waveBalances}). We then use an argument completely similar to the one used for the decay of positive waves in~\cite{Bressan}: if $I = [a,b]$ is an interval and $v_{i,\cont}^-(I)$ is too negative, then either $v^-_{i,\jump}$ is cancelled or the characteristics $a(t)$, $b(t)$ starting from $a$, $b$ collapse in a future time; in the last case the time $t$ for which $a(t) = b(t)$ is of the order of the length of the interval divided by the amount of negative wave, or by the interaction measure, cancellation measure and jump wave balance measure in the region spanned in time by characteristics from the interval. In this way we give a precise meaning to the inequality~\eqref{E:rough_balance}. We thus obtain the estimate~\eqref{E:decaneg2} first for the approximated wave-front tracking solution and for finitely many intervals, and then passing to the limit we recover the same estimate for the semigroup solution and for Borel sets (Lemmas~\ref{L:approxestimate},~\ref{L:decaEstOnIntervals}). Adding the already known decay estimate for positive waves, we obtain the desired result (Corollary~\ref{C:final_estimate}).


\section{General preliminary notions}
\label{S:gene_prel_notion}

Consider the Cauchy problem~\eqref{E:consLaw}. The following assumptions are done:

\begin{enumerate}

\item 
Strict hyperbolicity: we set $A(u)=\Dif f(u)$ and we assume that the eigenvalues $\{\lambda_{i}\}_{i}^{}$ of $A$ satisfy
\[
\lambda_1(u) <\dots < \lambda_{\dmns}(u),
\qquad u\in\Omega.
\]
We denote the unit right eigenvectors, and the left ones satisfying $r_{i}\cdot l_{j}=\delta_{ij}$, respectively by
\[
r_1(u) ,\dots ,r_{\dmns}(u)
\qquad
l_1(u) ,\dots ,l_{\dmns}(u).
\]

\item
Each $\bar \imath$-th characteristic field is either genuinely non-linear, i.e.~
\begin{equation*}
\label{E:genNonl}
|\Dif\lambda_{\bar \imath}(u) r_{\bar \imath}(u)|\geq k > 0,
\qquad
u\in\Omega,
\addtocounter{equation}{1}
\tag{$\theequation:{\bar\imath}$GN}
\end{equation*}
or linearly-degenerate, i.e.~
\begin{equation}
\label{E:linDeg}
\Dif\lambda_{\bar \imath}(u) r_{\bar \imath}(u)= 0,
\qquad
u\in\Omega.
\addtocounter{equation}{1}
\tag{$\theequation:{\bar\imath}$LD}
\end{equation}
\end{enumerate}
By the general theory on hyperbolic systems of conservation laws, one has then the following theorem.
\begin{theorem}
\label{T:sigmaLambda}
Let $\bar u=u^{-}\chi_{\{x<0\}}+ u^{+}\chi_{x\geq 0}$.
Then there exists a unique self-similar weak solution whose shocks satisfy the Lax compatibility condition:
\[
\lambda_{i}(\omega^{-})\geq\lambda_{i}(\omega ^{-}, \omega ^{+})\geq\lambda_{i}(\omega ^{+})
\qquad\text{at each jump $[\omega^{-},\omega^{+}]$ having speed $\lambda_{i}(\omega ^{-}, \omega ^{+}) $}.
\]
\end{theorem}
\begin{proof}
We recall just that one can define $C^{1}$-curves $\Psi_{i}(\sigma)(u_{0})$ and the scalars $\lambda_{i}(u_{0}, \Psi_{i}(\sigma)u_{0}) $ by
\[
\begin{cases}
\dot\Psi_{i}(\sigma)=\tilde r_{i}(\Psi_{i}(\sigma))
&\sigma\geq 0
\\
f(\Psi_{i}(\sigma))-f(u_{0})=\lambda_{i}(u_{0}, \Psi_{i}(\sigma))\big(\Psi_{i}(\sigma)-u_{0}\big)
& \sigma< 0
\end{cases},
\qquad
\Psi_{i}(0)(u_{0})=u_{0},
\]
where $\tilde r_{i}$ is a vector parallel to $r_{i}$ and satisfying $\Dif \lambda_{\bar\imath}\tilde r_{\bar\imath}=1$ if~\eqref{E:genNonl} holds, otherwise $|\tilde r_{\bar\imath}|=1$ if instead~\eqref{E:linDeg} holds.
The proof (see e.g.~Th.~5.3,~\cite{Bressan}) is then based on the inverse function theorem applied to the local $C^1$ homeomorphism $\Lambda:U(0)\ni\sigma\mapsto u^+\in U(u^-)$
\[\Lambda(\sigma_1,\dots,\sigma_\dmns)(u^-)=\Psi_\dmns(\sigma_\dmns)\circ\dots\circ \Psi_1(\sigma_1)(u^-) .
\]
It is applied in a domain where the Jacobian of the map is uniformly bounded away from $0$.
From the $C^{1}$-regularity of the curves $\Psi_{i}$, its Jacobian at $\sigma= 0$ is
$\big[\tilde r_1|\dots |\tilde r_\dmns\big](u^{-})=\Dif\Lambda(\sigma=0)(u^{-})
$.
Under~\eqref{E:genNonl}, if $u^{+}=\Psi_{\bar\imath}(\sigma)(u^{-})$ with $\sigma<0$ we have a \emph{compressive $\bar\imath $-shock}, otherwise a \emph{centered $\bar\imath $-rarefaction wave} (which is continuous).
In the case of linear degeneracy~\eqref{E:linDeg} we have an \emph{$\bar\imath$-contact discontinuity}.
\end{proof}

\begin{theorem}
\label{T:semigroup_uniqueness}
There exists a closed domain $\mathcal D\subset L^{1}(\R;\Omega)$ and a unique distributional solution $u=u(t,x)=[u(t)](x)$ which is a Lipschitz semigroup $\mathcal D\times [0,+\infty)\to \mathcal D $ and which for piecewise constant initial data coincides, for a small time, with the solution of the Cauchy problem obtained piecing together the standard entropy solutions of the Riemann problems.
Moreover, it lives in the space of $\BV$ functions.
\end{theorem}

Below, we will refer to $u$ as the semigroup solution, or equivalently as vanishing viscosity solution.
When referring to pointwise values of $u$, we tacitly take its $L^{1}$-representative such that the restriction map $t\mapsto u(t)\in L^{1}(dx)$ is continuous from the right and $u(t)$ is pointwise continuous from the right in $x$.\label{lab:normalizationu}

\subsection{\texorpdfstring{Decomposition of $u_{x}$ into wave measures}{Decomposition into wave measures}}
\label{Ss:decowave}

By the smallness assumption on the $\BV$ norm of the initial datum, one can assume that the eigenvalues of $A= \Dif f $ satisfy on $\Omega$
\[
\inf \lambda_{1}\leq \sup \lambda_{1}\  < \ \dots\  <\  \inf \lambda_{\dmns}\leq \sup \lambda_{\dmns} ,
\]
and the eigenvectors lie in different cones.
We then decompose as in the literature (see e.g.~pp.~93,201 of~\cite{Bressan}) the vector-valued measures $u_{x}$, $f(u)_{x}$ along the right eigenvectors of $A=\Dif f$. We adopt the following notation.
\begin{enumerate}
\label{eigenvectors}
\item
\label{eigenvectors:point1}
Under~\eqref{E:genNonl} we normalize the $\bar\imath $-th right eigenvector of $A(u)$ so that $\Dif\lambda_{\bar\imath} r_{\bar\imath} =1$, as in the proof of Theorem~\ref{T:sigmaLambda}: denote it by  $ \tilde r_{\bar\imath}(u) $.
Under~\eqref{E:linDeg} just take $\tilde r_{\bar\imath}(u)= r_{\bar\imath}(u)$, so that $| \tilde r_{\bar\imath} |=1$.

\item
We fix the left eigenvectors $\tilde l_{1}(u),\dots, \tilde l_{\dmns}(u)$ so that $\tilde l_{i}\cdot\tilde r_{j} =\delta_{ij}$ still holds for $i,j\in\{1,\dots,\dmns\}$.

\item
Given two values $u^{\pm}\in\Omega$, by the solution to the Riemann problem, briefly recalled in Theorem~\ref{T:sigmaLambda}, there exists $\sigma=(\sigma_{1},\dots,\sigma_{\dmns})\in\R^{\dmns}$ such that $u^{+}=\Lambda(\sigma)(u^{-})$.
We introduce the values
\begin{equation}
\label{E:omegai}
\omega_{0}=u^{-},
\qquad
\omega_{i}=\Psi(\sigma_{i})(\omega_{i-1})
\ i=1,\dots,\dmns.
\end{equation}
We define $\lambda_i(u^+,u^-)$ as the $i$-th eigenvalue of the averaged matrix
\[
A(\omega_{i},\omega_{i-1}) = \int_0^1 A\big(\theta \omega_{i} + (1-\theta)\omega_{i-1}\big)d\theta
\]
and $\tilde l_i(u^+,u^-)$, $\tilde r_i(u^+,u^-)$ vectors satisfying $\tilde l_i\cdot\tilde r_j =\delta_{ij} $ which are
\begin{itemize}
\item if $\sigma_{\bar\imath}<0$ and~\eqref{E:genNonl} holds, left$\backslash$right $\bar\imath $-eigenvectors of $A(\omega_{\bar\imath},\omega_{\bar\imath-1})$ normalized so that
\begin{subequations}
\label{E:sigmalu}
\begin{equation}
\label{E:sigmalushock}
\lambda_{\bar\imath}(\omega_{\bar\imath})-\lambda_{\bar\imath}(\omega_{\bar\imath-1})= \tilde l_{\bar\imath} \cdot(\omega_{\bar\imath}-\omega_{\bar\imath-1}).
\end{equation}
\item otherwise, $\tilde l_{i}(u^+,u^-)=\int_{0}^{1} \tilde l_{i}(\theta\omega_{i}+(1-\theta)\omega_{i-1}) $.
Then\footnote{%
Denote by $\sigma(u^{+};u^{-})$ the inverse function of $\sigma\mapsto u^{+}=\Lambda(\sigma;u^{-})$.
By the inverse function theorem $J\sigma(u^{+};u^{-})=(D\Lambda)^{-1}(\sigma(u^{+},u^{-});u^{-})$. If either the $i$-th component $\sigma_{i}$ is nonnegative or the $i$-th characteristic field is linearly degenerate, by the ODE satisfied by $\Psi_{i}$ one finds $\nabla\sigma_{i}(u^{+};u^{-})=\tilde l_{i}(u^{+})$ at least when $u^{+}=\Psi_{i}(\sigma_{i})(u^{-})$, being in that case $\dot\Psi_{i}(\sigma_{i})=\tilde r_{i}(\Psi_{i}(\sigma_{i}))$. Then by the fundamental theorem of calculus and the definition of $\omega_{0},\dots,\omega_{\dmns}$
\begin{align*}
\sigma_i(u^+,u^-) 
=\sigma_i(\omega_i,\omega_{i-1})
&= \int_0^1 \frac{d}{d\theta}\sigma_i\big(\omega_{i-1} + \theta(\omega_i-\omega_{i-1}), \omega_{i-1}\big)d\theta
\\
&= \int_0^1 \Big\{\tilde l_i\big(\omega_{i-1} + \theta(\omega_i-\omega_{i-1}), \omega_{i-1}\big) \cdot(\omega_i-\omega_{i-1}) \Big\}d\theta
= \tilde l_i \cdot(\omega_i-\omega_{i-1}).
\end{align*}%
}
\begin{equation}
\sigma_{i}= \tilde l_{i} \cdot(\omega_{i}-\omega_{i-1}).
\end{equation}
\end{subequations}
\end{itemize}
We call \emph{strength} of the $i$-th wave the value
\[
\sigma_{i}
=\tilde l_{i}\cdot (\omega_{i}-\omega_{i-1}).
\]
In the genuinely non-linear case, by the parameterization choice it is equal to $\lambda_{\bar\imath}(\omega_{\bar\imath})-\lambda_{\bar\imath}(\omega_{\bar\imath-1}) $.
\begin{remark}
\label{R:RH}
Notice that, as in the proof of Theorem~\ref{T:sigmaLambda}, we defined $\lambda_{i}(u^{+},u^{-})$ as the Rankine-Hugoniot speed of the $i$-th wave of the Riemann problem $[u^-,u^+]$, in case it is a shock or a contact discontinuity: indeed since
\[f(\omega_{i})-f(\omega_{i-1})=\int_{0}^{1}(\Dif f(\omega_{i-1} +\theta(\omega_{i}-\omega_{i-1}))(\omega_{i}-\omega_{i-1}))d\theta=A(\omega_{i}, \omega_{i-1})(\omega_{i}-\omega_{i-1}),\]
then, if the Rankine-Hugoniot condition $f(\omega_{i})-f(\omega_{i-1})=\lambda_{i}(\omega_{i}-\omega_{i-1})(\omega_{i}-\omega_{i-1})$ holds, $\omega_{i}-\omega_{i-1}$ is an $i$-eigenvector of $A(\omega_{i}, \omega_{i-1})$ with eigenvalue $\lambda_{i}(u^{+},u^{-}) $. In case of rarefaction waves it is just an average speed, and we see below that in the wave front tracking construction one may choose this as the speed of the `artificial' jump, up to small perturbations.
\end{remark}
\end{enumerate}
We recall now the wave decomposition of $u_x$ into wave measures, and consequently of the flow $f(u)_x$. Let
\label{En:rllambda}
\begin{gather*}
\lambda_i(t,x) = \lambda_i\big(u(t,x^+), u(t,x^-)\big) .
\\
\tilde l_i(t,x)= \tilde l_i\big(u(t,x^+), u(t,x^-)\big) ,
\qquad\qquad
\tilde r_i(t,x) = \tilde r_i\big(u(t,x^+), u(t,x^-)\big) .
\end{gather*}
\begin{lemma}
Define the scalar measures $v_{i}:= \tilde l_i\cdot u_x$. Then the following decomposition holds
\begin{subequations}
\label{E:waveDec}
\begin{gather}
\label{E:waveDecx}
u_x = \sum_{i=1}^{\dmns} v_{i} \tilde r_i(t,x), 
\\ 
-u_{t}=f(u)_x =\sum_{i=1}^{\dmns} \lambda_i(t,x) v_{i} \tilde r_i(t,x).
\end{gather}
\end{subequations}
\end{lemma}
\begin{proof} 
The first equation is a direct consequence of the fact that $\tilde r_{1},\dots,\tilde r_{\dmns}$ are linearly independent and $\tilde l_{i}\cdot \tilde r_{j}=\delta_{ij}$.
For justifying the second one, denote by the structure of $\BV$ functions (Sect.~3.7 of~\cite{AFP})
\[
u_{x}=\big[\Difxa u+ \Difxc u\big]+\sum_{h\in\N} \big[u(t,\gamma_{h}(t)^{+})-u(t,\gamma_{h}(t)^{-})\big] \Ll^{1}(dt),
\]
where $\Difxa u, \Difxc u $ are the absolutely continuous and Cantor part of $u_{x}$, while the sum is the jump part.
The values $u(t,\gamma_{h}(t)^{\pm})$ are indeed the values of the approximate jump at $(t,\gamma_{h}(t)^{})$ by a fine property of semigroup solutions recalled e.g.~in Theorem~\ref{T:pwconvergence} below.
Then by Volpert chain rule (Th.~3.99 of~\cite{AFP})
\begin{gather*}
\begin{split}
f(u)_x
&= \Dif f(u)\big[\Difxa u+ \Difxc u\big] + \sum_{h\in\N}\big[f(u(t,\gamma_{h}(t)^{+}))-f(u(t,\gamma_{h}(t)^{-}))\big] \Ll^{1}(dt)\restr{\Dom(\gamma_{h})}
\\
&\stackrel{\eqref{E:omegai}}{=} A(u)\big[\Difxa u+ \Difxc u\big] + \sum_{h\in\N}
\sum_{i=1}^{\dmns} 
\big[f(\omega_{i}(t,\gamma_{h}(t)^{+}))-f(\omega_{i-1}(t,\gamma_{h}(t)^{-}))\big] \Ll^{1}(dt)\restr{\Dom(\gamma_{h})}
\\
&\stackrel{\eqref{E:waveDecx}}{=}
\sum_{i=1}^{\dmns} \lambda_i(t,x)( v_{i})_{\cont} \tilde r_i(t,x) + \sum_{i=1}^{\dmns} \lambda_i(t,x)( v_{i})_{\jump} \tilde r_i(t,x)
=\sum_{i=1}^{\dmns} \lambda_i(t,x) v_{i} \tilde r_i(t,x),
\end{split}
\end{gather*}
where we also applied in the last step that each $\tilde r_{i}(u)$ is a right eigenvector of $A(u)$, and the Rankine-Hugoniot conditions of Remark~\ref{R:RH} for the jump part.
\end{proof}

\section{\texorpdfstring{Main $\SBV$ regularity argument}{Main SBV regularity argument}}
\label{S:mainArg}

Given a semigroup solution $u$, the wave decomposition~\eqref{E:waveDec} of $u_{x}$ along $\tilde r_{1},\dots,\tilde r_{\dmns}$ reduces the vectorial problem to scalar ones. In order to check that there is no Cantor part in $ u_{x}$, by~\eqref{E:waveDec} we will indeed check that there is no Cantor part in each $v_{i}$.

The argument here is based on piecewise constant front-tracking approximations of the semigroup solution $u$, some recall is provided in Section~\ref{S:frontTrackingApp}.
The magnitude of waves of each $\nu$-front-tracking approximation may change in time only when discontinuity lines meet, and waves interact. In~\cite{Bressan}, Section 7.6, an approximate conservation principle for wave strengths is expressed by introducing two finite measures concentrated on interaction points of physical waves: the \emph{interaction} and \emph{interaction-cancellation} measures
\begin{equation}
\label{E:intrInterCanc}
\mu_{\nu}^{I}(\{P\})=|\sigma' \sigma''|
\qquad
\mu_{\nu}^{IC}(\{P\})=|\sigma'\sigma''|+\begin{cases} |\sigma'|+|\sigma''|-|\sigma'+\sigma''| & i=i'\\ 0 & i\neq i' \end{cases} .
\end{equation}
Above $\sigma'$, $\sigma''$ denote the incoming strengths of the two physical waves interacting at $P$.\\
\begin{figure}[ht!] 
   \centering
   \def\svgwidth{.6\columnwidth} 
   \input{balancearxiv.pdf_tex}
\caption[Balance of waves.]{The yellow area represents the region $\Gamma$ in the plane $(x,t)$, waves enter and exit.}
\label{fig:immagineBalance}
\end{figure}
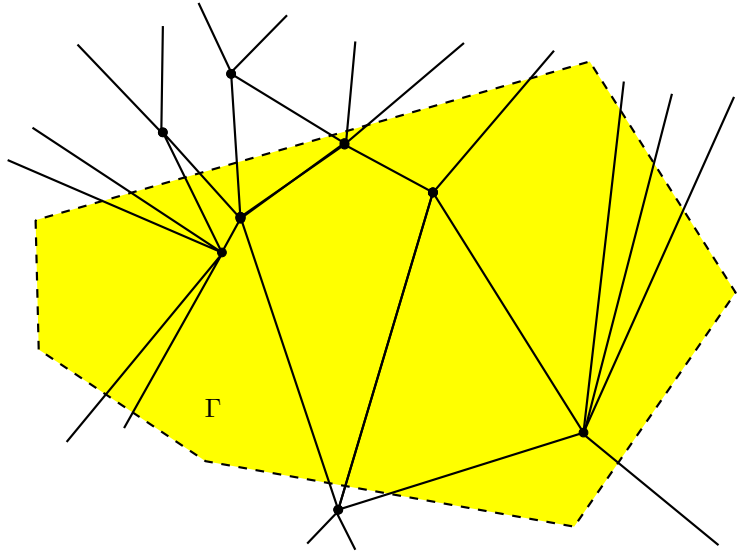
More precisely, given a polygonal region $\Gamma$ with edges transversal to the waves it encounters (Figure~\ref{fig:immagineBalance}), one considers the total amount $W^{i\pm}_{\inn}$, $W^{i\pm}_{\out}$ of positive and negative $i$-waves entering and exiting $\Gamma$. Considering the `fluxes' (the amount of positive$\backslash$negative $i$th-waves entering$\backslash$exiting $\Gamma$)
\[
W^{i}_{\inn}= W^{i+}_{\inn} - W^{i-}_{\inn}
\qquad
W^{i}_{\out}= W^{i+}_{\out} - W^{i-}_{\out},
\qquad i=1,\dots,N ,
\]
then one has the upper bounds
\[
|W^{i}_{\out}- W^{i}_{\inn}| \leq C\mu^{I}(\Gamma)
\qquad
|W^{i\pm}_{\out}- W^{i\pm}_{\inn}| \leq C\mu^{IC}(\Gamma) .
\]

Since the above measures satisfy uniform estimates w.r.t.~$\nu$, by compactness one can define measures $\mu^{I}$, $\mu^{IC}$ on the entropy solution by taking $w^{*}$-limits of the corresponding measures on a suitable sequence of $\nu$-front-tracking approximations.

Assume that the $i$-th characteristic field is genuinely non-linear,~\eqref{E:genNonl}.
A balance of the kind above will lead in Lemma~\ref{L:decaEstOnIntervals} below to the following estimate on a semigroup solution $u$: the density w.r.t.~time of the continuous part $(v_{\bar\imath})_{\cont}$ of $v_{\bar\imath}=\tilde l_{\bar\imath}\cdot u_{x}$ is controlled at time $s$ by the Lebesgue measure and by a finite measure on a horizontal strip around $s$.
This estimate is analogous to the decay of positive waves proved in Theorem 10.3 of~\cite{Bressan}.
We state it here, postponing the proof to Section~\ref{S:mainEstimate}.
\begin{theorem}
\label{T:continuousEstimate}
Suppose~\eqref{E:genNonl} holds for the $\bar\imath$-th characteristic field.
Denoting
\[
v_{\bar\imath}=\int \big[v_{\bar\imath}(t) \big]dt,
\qquad
v_{\bar\imath}(t)=(v_{\bar\imath}(t))_{\cont}+(v_{\bar\imath}(t))_{\jump}:
\quad\text{$(v_{\bar\imath}(t))_{\jump}$ purely atomic,} 
\]
then there exists a finite, nonnegative Radon measure $\mu_{\bar\imath}^{ICJ}$ on $\R^{+}\times\R$ such that for $s>\tau>0$
\begin{align}
\label{E:decayEst}
\OO\Big\{\Ll^{1}(B)/\tau+\mu_{\bar\imath}^{ICJ}([s-\tau,s+\tau]\times\R) \Big\}
&\geq 
|(v_{\bar\imath}(s))_{\cont}|\big(B\big)
\qquad
\forall B\text{ Borel subset of }\R
.
\addtocounter{equation}{1}
\tag{$\theequation:{\bar\imath}$}
\end{align}
\end{theorem}

The Radon measure $\mu_{\bar\imath}^{ICJ}$ is defined in the next section as a $w^{*}$-limit of measures~\eqref{E:muICJinu} relative to front-tracking approximations.
It takes into account both general interactions and cancellations and also balances for exclusively the jump part of $v_{\bar \imath}$.
In the statement $v_{\bar\imath}(s)$ is the representative continuous from the right, coinciding with $\tilde l_{i}\cdot\big(u(s))_{x}$.

Following an argument already in~\cite{ambDel}, the estimate~\eqref{E:decayEst} allows to prove that $v_{\bar \imath}$ is made by merely a jump part concentrated on the graphs of countably many time-like Lipschitz curves and a part absolutely continuous w.r.t.~the Lebesgue measure.
Indeed, $(v_{\bar\imath})_{\cont}$ is the integral w.r.t.~$\Ll^{1}(dt)$ of $(v_{\bar\imath}(t))_{\cont}$: if $|(v_{\bar\imath}(s))_{\cont}|\big(B\big)>0$ at some time $s$ for an $\Ll^{1}$-negligible set $B$, then by~\eqref{E:decayEst} the time marginal of $\mu^{ICJ}_{\bar\imath}$ has an atom at $s$, which may happen only countably many times.

\begin{corollary} 
\label{C:viSBV}
Let $u$ be a semigroup solution of the Cauchy problem for the strictly hyperbolic system~\eqref{E:consLaw}.
Consider the $\bar\imath$-wave measure $v_{\bar\imath}=\tilde l_{\bar\imath}\cdot u_{x}$. If~\eqref{E:genNonl} holds, then $v_{\bar \imath}$ has no Cantor part.
\end{corollary}

If all the characteristic fields are genuinely non-linear, by the wave decomposition~\eqref{E:waveDec} the above estimate yields then $\SBV_{}([0,T]\times\R;\R^{\dmns})$ regularity of $u$ for all $T>0$.

\begin{corollary}
\label{C:mainSBV}
Let $u$ be the semigroup solution of the Cauchy problem
\[
\begin{cases}
u_t + f(u)_x = 0, 
\\
u(t=0)=\bar u
\end{cases}
\qquad u :\R^+\times\R\to \Omega\subset\R^\dmns,
\qquad f\in C^{2}(\Omega;\R^{\dmns})
\]
for a strictly hyperbolic system of conservation laws where each characteristic field is genuinely non-linear, with initial datum $\bar u$ small in $\BV(\R;\Omega)$.
Then $u(t)\in\SBV(\R;\Omega)$ out of at most countably many times.
\end{corollary}

\begin{proof}[Proof of Corollaries~\ref{C:viSBV} and~\ref{C:mainSBV}]
By the general theory on $1$-dimensional systems of conservation laws, $u$ belongs to $\BV_{}([0,T]\times\R;\Omega)$ for all $T>0$.
By the structure theorem of $\BV$ functions (see~Sect.~3.7 of~\cite{AFP}), the derivative of $u$ can be decomposed into a jump part, concentrated on a $1$-countably rectifiable set, a part which is absolutely continuous w.r.t.~the Lebesgue measure, and a remaining part---the Cantor part---which is singular w.r.t.~$\Ll^{2}$ and vanishes on sets having finite Hausdorff $1$-dimensional measure:
\[
\Dif u= \nabla u\,\Ll^{2} + \Difc u + \Difj u .
\] 
We want to prove that the Cantor parts $\Diftc u$, $\Difxc u$ of the both components $u_{t},u_{x}$ of $\Dif u$ vanish.

Denote by $u(t)$ the $1$-space-dimensional restriction of $u$ at time $t$, which is a function of $x$.
By the slicing theory of $\BV$ functions (Theorems~3.107-108 in~\cite{AFP}), not only
\begin{equation}
\label{E:slice1}
u_{x}=\int u(t)_{x} dt ,
\end{equation}
but moreover the Cantor part of $u_{x}$ is the integral, w.r.t.~$\Ll^{1}(dt)$, of the Cantor part of $u(t)$: one has also the disintegration
\[
\Difxc u=\int \Difc\big(u(t)\big) dt.
\]
Since, moreover, by combining Volpert chain rule (Th.~3.99 of~\cite{AFP}) and the conservation law one has
\[
\Diftc u = f'(u) \Difxc u ,
\]
it suffices to show that for $\Ll^{1}$-a.e.~$t$ the function $u(t)$ belongs to $\SBV(\R;\R^{\dmns})$.

If one sets $v_{i}(t)= \tilde l_i\cdot u(t)_x $, then by the slicing~\eqref{E:slice1} and as in the wave decomposition~\eqref{E:waveDec} one derives
\[
v_{i} = \int v_{i}(t) dt
\qquad i=1,\dots,\dmns,
\qquad
u(t)_x=\sum_{i=1}^{\dmns} \tilde r_{i}(t,x) v_{i}(t) .
\]
As a consequence, there is a Cantor part at time $s$ in some $v_{\bar\imath}(s)$ precisely when $u(s)_{x}$ has a Cantor part.

It remains then to prove that~\eqref{E:genNonl} implies that $v_{\bar\imath}(s)$ has no Cantor part for $\Ll^{1}$-a.e.~$s$.
By Theorem~\ref{T:continuousEstimate}, the assumption of genuine nonlinearity~\eqref{E:genNonl} implies that the estimate~\eqref{E:decayEst} holds for the $\bar\imath$-th characteristic field.
The fact that $v_{\bar\imath}(s)$ has a Cantor part means that there exists an $\Ll^{1}$-negligible compact set $K$ with $v_{\bar\imath}(s)(K)>0$ and with no atom of $v_{\bar\imath}(s)$.
Then for all $s>\tau>0$
\begin{align*}
0< |v_{\bar\imath}(s)|(K)
&\stackrel{\phantom{\eqref{E:decayEst}}}{=} 
 |(v_{\bar\imath}(s))_{\cont}|(K)
\\
&\stackrel{\eqref{E:decayEst}}{\leq} 
\OO\Big\{
	\Ll^{1}(K)/\tau
	+\mu_{\bar\imath}^{ICJ}\big([s-\tau,s+\tau]\times\R \big) \Big\} 
	.
\end{align*}
As we are taking $\Ll^{1}(K)=0$, by outer regularity of Borel measures when $\tau\downarrow0$ this means that
\[
\mu_{\bar\imath}^{ICJ}\big(\{s\}\times\R \big)>0.
\]
It is indeed true that if $\Ll^{1}(K)=0$ then $\mu_{\bar\imath}^{ICJ}(\{s\}\times K)\geq  \OO|v_{\bar\imath}(s)_{\cont}|(K)>0$.
As the above measures are locally finite, this can thus happen at most countably many times: for all other times $t$ the continuous part of $v_{\bar\imath}(t)$ is absolutely continuous.

As a consequence, if~\eqref{E:genNonl} holds for all $\bar\imath=1,\dots,\dmns$ then $u(t)\in\SBV(\R;\R^{\dmns})$ out of countably many times, the times when the time marginal of anyone of the various $\mu^{ICJ}_{i}$ has a jump, $i=1,\dots,\dmns$. 
This yields the membership of $u$ in $\SBV_{\loc}(\R^{+}\times\R;\R^{\dmns})$.
\end{proof}

\section{Recalls on the approximation by front-tracking solutions}
\label{S:frontTrackingApp}
We recall in this section a result about the convergence of a suitable sequence of $\nu$-approximate front-tracking solutions (Pages~219-220 in~\cite{Bressan}).

The $\nu$-approximate front-tracking solutions $\{u^{\nu}\}_{\nu}$ are, roughly speaking, piecewise-constant functions obtained approximating by a step function the initial data, (approximatively) solving Riemann problems at discontinuity points and piecing together these solutions until the time they interact and discontinuity lines cross each other: at that time the procedure starts again.
By the construction, which allows small perturbations of the speed, only two discontinuity lines are allowed to cross at one time.
Each outgoing $i$-rarefaction wave $[u^{-}, u^{+}= \Psi_i(\sigma)(u^{-})]$ in the approximate solution of the Riemann problems is decomposed into small jumps $[\omega_{h}, \omega_{h+1}=\Psi_i(\sigma_{h})(\omega_{h})]$ of strength $\sigma_{h}$ at most $\nu$; among the various possibilities, we let the $h$-th jump $[\omega_{h}, \omega_{h+1}]$ travel with the mean speed $\lambda_{i}(u^{+},u^{-})$.
In order to control the number of discontinuity lines, if the interacting wave fronts are small enough a simplified Riemann solver is used, which leaves unchanged the size of the incoming waves introducing a non-physical wave front traveling with fixed speed higher then $\lambda_{\dmns}$; the total size of non-physical waves is controlled at each time by a constant $\varepsilon_{\nu}$.

At each time $t$, the restriction $u^{\nu}(t)$ is a step function: its derivative consists of finitely many deltas.
Below one can see that if one fixes suitable thresholds, it is possible to group these deltas in two families: up to subsequences, the largest of them converge in the $\nu$-limit to the jumps of the entropy solution $u^{}$ at time $t$, for $\Ll^{1}$-a.e.~time, while the others tail off up to the remaining continuous part of $u(t)_{x}$.
We distinguish the jumps, excluding interaction times, depending on their characteristic family.
\\
Looking at the $(t,x)$-variables instead of time-restrictions of $u^{\nu}$, the derivative of $u^{\nu}$ is concentrated on polygonal lines and consists only of the jump part.
Nevertheless, these broken lines can be grouped as follows in order to distinguish those converging to the jump set of $u$ and those part of $u^{\nu}_{x}$ converging to the continuous part of $u_{x}$.

\begin{definition}[Maximal $(\varepsilon_{0},\varepsilon_{1})$-shock front]
\label{D:max_epsilon_shock}
A \emph{maximal $(\varepsilon_{0},\varepsilon_{1})$-shock front of the $i$-th family of a $\nu$-approximate front-tracking solution $u^{\nu}$} is any maximal (w.r.t.~inclusion) polygonal line $\big(t,\gamma(t)\big)$ in the $(t,x)$-plane, with ${t^{-}\leq t\leq t^{+}}$,
\begin{itemize}
\item[-]
whose segments are $i$-shocks of $u^{\nu}$ with strength $|\sigma|\geq \varepsilon_{0}$, and at least once $|\sigma|\geq \varepsilon_{1} >\varepsilon_{0}$;
\item[-]
whose nodes are interaction points of $u^{\nu}$;
\item[-]
which is on the left of any other polygonal line it intersects and having the above properties.
\end{itemize}
\end{definition}
The family of $i$-maximal $(\varepsilon_{0},\varepsilon_{1})$-shock fronts of $u^{\nu}$ is totally ordered by the displacements of the polygonal lines on $\R\times\R^{+}$, and, up to extracting a subsequence, we are allowed to assume that its cardinality is a constant $M^{i}_{(\varepsilon_{0},\varepsilon_{1})}$ independent of $\nu$.
Let $\gamma^{\nu,i}_{(\varepsilon_{0},\varepsilon_{1}),m}:[t^{\nu,i,-}_{(\varepsilon_{0},\varepsilon_{1}),m},t^{\nu,i,+}_{(\varepsilon_{0},\varepsilon_{1}),m}]\to\R$ denote the uniformly Lipschitz paths of $i$-maximal $(\varepsilon_{0},\varepsilon_{1})$-shock fronts in $u_\nu$ and consider their graphs
\[
\mathcal J^{\nu,i}_{(\varepsilon_{0},\varepsilon_{1})}
=
\bigcup_{m=1}^{M^{i}_{(\varepsilon_{0},\varepsilon_{1})}} \Graph\big( \gamma^{\nu,i}_{(\varepsilon_{0},\varepsilon_{1}),m}\big).
\]
Notice that the set $\mathcal J^{\nu,i}_{(\varepsilon_{0},\varepsilon_{1})} $ enlarges as $(\varepsilon_{0},\varepsilon_{1})$ goes to $0$.

Consider sequences $0<2^{k}\varepsilon_{0}^{k} \leq \varepsilon_{1}^{k} \downarrow 0 $, where $\varepsilon_{0}^{k} $ is not necessarily a power but has just an apex $k$.
Up to subsequences and a diagonal argument, by a suitable labeling of the curves one can assume that for each $i,k,m$ fixed the Lipschitz curves $\gamma ^{\nu,i}_{(\varepsilon_{0}^{k},\varepsilon_{1}^{k}),m}$ converge uniformly
\begin{subequations}
\label{EG:jumpsconv}
\begin{equation}
\label{E:shocksconv}
t^{\nu,i,\pm}_{(\varepsilon_{0}^{k},\varepsilon_{1}^{k}),m}
\to 
t^{i,\pm}_{(\varepsilon_{0}^{k},\varepsilon^{k}_{1}),m} ,
\qquad
\gamma ^{\nu,i}_{(\varepsilon^{k}_{0},\varepsilon^{k}_{1}),m}
\to
 \gamma^{i}_{(\varepsilon_{0}^{k},\varepsilon^{k}_{1}),m}
\qquad
\text{as $\nu\to\infty$}
\end{equation}
to \emph{distinct} Lipschitz curves which cover the jump set of $u$: out of countably many points in the $(x,t)$-plane, either $u$ is continuous, and equal to the pointwise limit of $u^{\nu}$ at that point, or has jump.
In the case of a jump, the jump point belongs to the graph of some $\gamma^{i}_{(\varepsilon_{0}^{k},\varepsilon^{k}_{1}),m}$ for a suitable triple $i,k,m$; moreover, the left$\backslash$right limits of the front-tracking approximation at the jump curve $\gamma ^{\nu,i}_{(\varepsilon^{k}_{0},\varepsilon^{k}_{1}),m} $ converge, for $m,k$ fixed and $\nu\to\infty$, to the left$\backslash$right limits of $u$ at that jump.
Below one finds the precise statement.

\begin{theorem}[Th.~10.4 in~\cite{Bressan}]
\label{T:pwconvergence}
The jump part of $u$ is concentrated on the graphs of  $\gamma ^{i}_{(\varepsilon_{0}^{k},\varepsilon_{1}^{k}),m}$,
\[
\mathcal J = \bigcup_{i=1}^{\dmns} \bigcup_{ m, k}\Graph\big(\gamma ^{i}_{(\varepsilon^{k}_{0},\varepsilon^{k}_{1}),m} \big).
\]
Moreover,  $u$ is continuous and equal to the pointwise limit of $u^\nu$ out of $\mathcal J$. Define the countable set
\begin{align*}
\Theta
=& \Theta_{0}\cup \Theta_{1} \cup \Theta_{2}\cup\Theta_{3}
\\
=& \Big\{\text{jump points of the initial datum $\bar u$} \Big\}
 \bigcup \Big\{\text{atoms of $\mu^{IC}$} \Big\}
\\
&\qquad\bigcup\Big\{\text{intersection of any $\gamma ^{i}_{(\varepsilon_{0}^{k},\varepsilon_{1}^{k}),m} $ and $\gamma ^{j}_{(\varepsilon_{0}^{k},\varepsilon_{1}^{k}),m} $ with $i\neq j$} \Big\} 
 \bigcup \Big\{\text{endpoints of $\gamma ^{i}_{(\varepsilon^{k}_{0},\varepsilon^{k}_{1}),m}$}\Big\}
.
\end{align*}
Then\footnote{For obtaining the central term below we refer more precisely to formulas (10.78-79) in~\cite{Bressan} for the genuinely non-linear case, and to an adaptation of the proof after the absurd (10.83) in~\cite{Bressan} for the linearly-degenerate case.
These formula show also that the curves in~\eqref{E:shocksconv} may intersect only at endpoints for $k$ fixed, as we claimed before the theorem.
For the purpose of this paper only the genuinely non-linear statement will be relevant, thus we have not specified that the curves $\gamma^{\nu}$ are actually defined in a different way for linearly degenerate characteristics.},
at each point $(t,\gamma(t))=\lim_{\nu}(t_{\nu},\gamma^{\nu_{}}(t_{\nu}))$ of $\mathcal J\setminus\Theta$ one has a jump between the distinct values
\begin{align}
\label{E:limitsconv1}
&
u(t, \gamma(t)^{-})
=
\lim_{\nu} u^{\nu_{}}(t_{\nu}, \gamma^{\nu_{}}(t_{\nu})^{-}) 
=
\lim_{\substack{(s,y)\to(t, \gamma(t)) \\y< \gamma(s)}} u(s,y) 
,
\\
\label{E:limitsconv2}
&
u(t, \gamma(t)^{+})
=\lim_{\nu} u^{\nu_{}}(t_{\nu}, \gamma^{\nu_{}}(t_{\nu})^{+}) 
=
\lim_{\substack{(s,y)\to(t, \gamma(t)) \\y>\gamma(s)}} u(s,y) 
.
\end{align}
\end{theorem}
\end{subequations}

We now stress some consequences of the above analysis.
On one hand, the fine approximation of $u$ yields in turn piecewise-constant approximations of any function depending on time-space through sufficiently smooth composition with $u$, and depending on $u^{\pm}$ at jump points.
The countably many points of $\Theta$ are not considered in these approximations statement, and not relevant.
In particular, Theorem~\ref{T:pwconvergence} yields the following corollary relative to the functions $\tilde l_i$, $\tilde r_i$, $\lambda_i$ introduced at Page~\pageref{En:rllambda}.

\begin{subequations}
\begin{notation}
When adding an apex $\nu$ we refer to the corresponding quantities relative to the $\nu$-front tracking approximation $u^{\nu}$: at each point $\lambda_{i}(t,x):=\lambda_{i}(u(t,x^{+}),u(t,x^{-}))$ and so on, while $\lambda_{i}^{\nu}(t,x):=\lambda_{i}(u^{\nu}(t,x^{+}),u^{\nu}(t,x^{-}))$ and similarly
\label{E:lambdalrnu}
\begin{align}
&
\tilde l_i^\nu (t,x) := \tilde l_i\big(u^{\nu}(t,x^{+}),u^{\nu}(t,x^{-})\big),
&&
\tilde r_i^\nu (t,x) := \tilde r_i\big(u^{\nu}(t,x^{+}),u^{\nu}(t,x^{-})\big).
\end{align}
\end{notation}
\begin{remark}
\label{R:barlambda}
Front-tracking approximations do not satisfy in general the precise Rankine-Hugoniot conditions, but they may have a small error.
Then we define the functions, for $i=1,\dots,\dmns$,
\begin{equation}
\bar\lambda_{i}^{\nu}(t,x)=
\begin{cases}
\lambda_{i}(u^{\nu}(t,x)) 
&\text{at points which are not on $i$-fronts of $u^{\nu}$} 
\\
\dot{\gamma ^\nu}(t) 
&\text{at discontinuity points of the $i$-th characteristic family}
\end{cases}
\end{equation}
which are the actual speed of propagation of the outgoing jump of $u^{\nu}$ if belonging to the corresponding characteristic family.
In general $\lambda_{i}^{\nu}(t, \gamma ^\nu(t)) \neq\dot{\gamma ^\nu}(t) $ and the difference vanishes as $\nu\to\infty$.
\end{remark}
\end{subequations}

\begin{corollary}
\label{C:corAprroximation}
At each point out of $\Theta\cup\mathcal J$ the functions $\tilde l_i$, $\tilde r_i$, $\lambda_i$ are the pointwise limits of their piecewise constant front-tracking approximations $\lambda_i^\nu$, $\tilde l_i^\nu$, $\tilde r_i^\nu$ introduced above.

Moreover, at each point $(t,\gamma(t))=\lim_{\nu}(t_{\nu},\gamma^{\nu_{}}(t_{\nu}))$ of $\mathcal J\setminus \Theta$ one has
\begin{gather*}
\tilde l_{i}(t, \gamma(t))=\lim_{{\nu}} \tilde l_i^{ \nu_{}} (t_{\nu},\gamma^{\nu_{}}(t_{\nu})), 
\qquad
\tilde r_i(t, \gamma(t))= \lim_{{\nu}} \tilde r_i^{ \nu_{}}(t_{\nu},\gamma^{\nu_{}}(t_{\nu})),
\\
\lambda_i(t, \gamma(t)) 
=\lim_{\nu_{}}\lambda_i^ {\nu}(t_{\nu},\gamma^{\nu_{}}(t_{\nu}))
=\lim_{\nu_{}} \dot{ \gamma ^{\nu}}(t).
\end{gather*}
\end{corollary}

\begin{proof}
Out of $\Theta\cup\mathcal J$ the statement follows just by the pointwise convergence of $u^{\nu}$ to $u$ provided by Theorem~\ref{T:pwconvergence}, because being $f\in C^{1}(\Omega;\R^{\dmns})$ the eigenvalues$\backslash$eigenvectors of $A(u)$ depend continuously on the variable $u$.
Also $\tilde l_i$, $\tilde r_i$, $\lambda_i$ depend continuously on the values $u^{\pm}$, being suitable means of continuous functions on the segment $[u^{-},u^{+}]$, and they depend on the time-space variables through the composition with the left and right limits $u$$\backslash$$u^{\nu}$ at the jump: the claim at a jump point $(t,\gamma(t))=\lim_{\nu}(t_{\nu},\gamma^{\nu_{}}(t_{\nu}))$ out of $\Theta$ follows then by the convergence of $u^{\nu_{}}(t_{\nu},\gamma^{\nu_{}}(t_{\nu})^{\pm})$ to $u^{\pm}(t,\gamma(t))$.
\end{proof}

\begin{corollary}
\label{C:hausdorfpendenze}
Each $i$-maximal $(\varepsilon^{k}_{0},\varepsilon^{k}_{1})$-shock front in $u_{\nu}$ converges (locally) as $\nu\to\infty$ to the corresponding graph of $\gamma ^{i}_{(\varepsilon^{k}_{0},\varepsilon^{k}_{1}),m} $ in the Hausdorff distance: $d_{H}\big(\Graph\gamma ^{\nu,i}_{(\varepsilon^{k}_{0},\varepsilon^{k}_{1}),m},\Graph \gamma ^{i}_{(\varepsilon^{k}_{0},\varepsilon^{k}_{1}),m}\big)\to 0$.
Moreover, there is a pointwise convergence also for the slopes $\dot{\gamma }^{\nu,i}_{(\varepsilon^{k}_{0},\varepsilon^{k}_{1}),m} (t)$ out of a countable set of $t$.
\end{corollary}
\begin{proof}
This follows directly by~\eqref{EG:jumpsconv} and Corollary~\ref{C:corAprroximation} taking into account that the $\nu$-front tracking approximate solutions (approximatively) satisfy on $i$-shocks the Rankine-Hugoniot conditions
\[
f(u^{+})-f(u^{-})=\lambda_{i}(u^{+},u^{-})(u^{+}-u^{-}).
\qedhere
\]
\end{proof}
On the other hand, Theorem~\ref{T:pwconvergence} allows handy approximations of the jump part of $u_{x}$, and therefore of the jump part of the wave measures $v_{i}$.
More precisely, consider for fixed $(\varepsilon_{0},\varepsilon_{1})$ the measures
\[
\label{E:sepjump}
w^{\nu,i,(\varepsilon_{0},\varepsilon_{1})}_{\jump} = u_{x}^{\nu}\llcorner_{\mathcal J^{\nu,i}_{(\varepsilon_{0},\varepsilon_{1})}},
\qquad
v^{\nu,i,(\varepsilon_{0},\varepsilon_{1})}_{\jump} = \tilde l_{i}\cdot w^{\nu,i,(\varepsilon_{0},\varepsilon_{1})}_{\jump}, 
\]
which are concentrated on finitely many segments, up to any finite time.

\begin{corollary}
\label{C:weakConvergences}
There exists a sequence $\nu_{k}$ such that
\[
\Difxj u
= \wlim_{k} \sum_{i=1}^{\dmns} w^{\nu_{k},i,(\varepsilon^{k}_{0},\varepsilon^{k}_{1})}_{\jump} ,
\qquad
(v_{i})_{\jump}
= \wlim_{k} v^{\nu_{k},i,(\varepsilon_{0}^{k},\varepsilon_{1}^{k})}_{\jump}  .
\]
Denoting by $\Difxa u_{x}=\nabla u\mathcal{L}^{2}_{\llcorner \R^{+}\times \R}$ the absolutely continuous part, $\Difxc u$ the Cantor part of $u_{x}$, one has
\[
\Difxa u+ \Difxc u= \wlim_{k} {w_{\cont}^{\nu_{k},(\varepsilon^{k}_{0},\varepsilon^{k}_{1})}},
\qquad
w_{\cont}^{\nu_{k},(\varepsilon^{k}_{0},\varepsilon^{k}_{1})}
:=
\big(u^{\nu_{k}} \big)_{x}- \sum_{i=1}^{\dmns} w^{\nu_{k},i,(\varepsilon^{k}_{0},\varepsilon^{k}_{1})}_{\jump}
.
\]
Similarly each $\lambda_{i} (v_{i})_{\jump}$ is the limit of $\bar\lambda_{i}^{\nu_{k}} v^{\nu_{k},i,(\varepsilon_{0}^{k},\varepsilon_{1}^{k})}_{\jump}$, while $\lambda_{i} (v_{i})_{\cont}$ the one of $\bar\lambda_{i}^{\nu_{k}}\big[v_{i}^{\nu_{k}}-v^{\nu_{k},i,(\varepsilon_{0}^{k},\varepsilon_{1}^{k})}_{\jump}]$.
\end{corollary}
\begin{remark}
\label{R:convrestr}
By inspection of the proof below, one can see that the convergence in the statement of Corollary~\ref{C:weakConvergences} holds in the same way for the jump and continuous part of $\big(u(t)\big)_{x}$, for each time $t$ except at most the ones corresponding to the points in $\Theta$ of Theorem~\ref{T:pwconvergence}.
\end{remark}
\begin{proof}
The $\BV$ structure of the semigroup solution $u$ (Sect.~3.7 of~\cite{AFP}) gives the formula for the jump part of $ \Dif u$: for $M^{i}_{k}=M^{i}_{(\varepsilon^{k}_{0},\varepsilon^{k}_{1})}$, it is the matrix-valued measure
\begin{align*}
\Difj u
&= \left(u^{+}-u^{-}\right) \otimes n\, \mathcal{H}^{1}_{\llcorner \mathcal J} 
=
\lim_{k\to\infty} 
\sum_{i=1}^{\dmns} 
\sum_{m=1}^{M^{i}_{k}} 
\left(u^{+}-u^{-}\right) \otimes n\, \mathcal{H}^{1}_{\llcorner \Graph\big( \gamma^{i}_{(\varepsilon^{k}_{0},\varepsilon^{k}_{1}),m}\big)}
.
\end{align*}
Moreover, each addend of the sum is the narrow limit of the analogous one relative to $u^{\nu}$ on the corresponding $(\varepsilon_{0}^{k}, \varepsilon_{1}^{k})$-shock front: by~\eqref{E:limitsconv1},~\eqref{E:limitsconv2}, Corollary~\ref{C:hausdorfpendenze} and the dominated convergence theorem
\begin{align*}
\sum_{m=1}^{M^{i}_{k}} 
\left(u^{+}-u^{-}\right) \otimes n\, \mathcal{H}^{1}_{\llcorner \Graph\big( \gamma^{i}_{(\varepsilon_{0}^{k},\varepsilon_{1}^{k}),m}\big)}
&=
  \wlim_{\nu} 
 \sum_{m=1}^{M_{k}^{i}} 
 \left(u^{\nu+}-u^{\nu-}\right) \otimes n^{\nu}\, 
 \mathcal{H}^{1}_{\llcorner \Graph\big( \gamma^{\nu,i}_{(\varepsilon_{0}^{k},\varepsilon_{1}^{k}),m} \big)}
 \\
& =
\wlim_{\nu} \left(u^{\nu+}-u^{\nu-}\right) \otimes n^{\nu}\, \mathcal{H}^{1}_{\llcorner {\mathcal J^{\nu,i}_{(\varepsilon_{0}^{k},\varepsilon_{1}^{k})}}} .
\end{align*}
Collecting the two limits, one has the expression of the $x$-component
\[
\Difxj u
=
\lim_{k}
\bigg(
\wlim_{\nu} \sum_{i=1}^{\dmns}w^{\nu,i,(\varepsilon_{0}^{k},\varepsilon_{1}^{k})}_{\jump}
\bigg)
.
\]
In particular, one can find a sequence $\nu_{k}$ such that
\[
\Difxj u
= \wlim_{k} 
\sum_{i=1}^{\dmns}w^{\nu_{k},i,(\varepsilon_{0}^{k},\varepsilon_{1}^{k})}_{\jump}
.
\]

Since $u^{\nu_{k}}$ converges to $u$ in $L^{1}_{\loc}(\R^{+}\times\R)$, then the $x$-derivative converges weakly$^{*}$.
One then takes the difference between the two terms $(u^{\nu_{k}})_{x}$ and $w^{\nu_{k},i,(\varepsilon^{k}_{0},\varepsilon^{k}_{1})}_{\jump}$, and takes their weak$^{*}$-limit, in order to see that each $w_{\cont}^{\nu_{k},(\varepsilon_{0}^{k},\varepsilon_{1}^{k})}$ weakly$^{*}$-converge to the continuous part of $u_{x}$, which is the difference between $u_{x}$ itself and its jump part.

The convergence of $v^{\nu,i,(\varepsilon_{0},\varepsilon_{1})}_{\jump}=\tilde l_{i}^{\nu_{k}} \cdot w^{\nu_{k},i,(\varepsilon^{k}_{0},\varepsilon^{k}_{1})}_{\jump}$ to $(v_{i})_{\jump}$ and of $\bar\lambda_{i}^{\nu} v^{\nu,i,(\varepsilon_{0},\varepsilon_{1})}_{\jump} $ to $\lambda_{i}^{}(v_{i})_{\jump} $ follow as above, taking into account also Corollary~\ref{C:corAprroximation}.
\end{proof}

\section{Main estimate}
\label{S:mainEstimate}

As explained in Section~\ref{S:mainArg}, the $\SBV$ regularity of $u$ amounts to the estimate
\begin{align*}
\label{E:decayEst2}
\exists C>0
\qquad
C\Big\{\Ll^{1}(B)/\tau+\mu_{i}^{ICJ}([s-\tau,s+\tau]\times\R) \Big\}
&\geq 
|(v_{i}(s))_{\cont}|\big(B\big)
\qquad
\forall \ s>\tau>0
,
\end{align*}
which shows that in the case of genuine non-linearity the non-atomic part $(v_{i}(s))_{\cont}$ of $v_{i}(s)$ is controlled by the Lebesgue $1$-dimensional measure and the interaction-cancellation-jump wave balance measure of a strip around the time $s$.

In proving it we consider the \emph{wave balance} and \emph{jump wave balance} measures
\begin{equation}
\label{E:balSources}
\mu_{i}^{}=\partial _{t} v_{i} +\partial_{x} (\lambda_{i} v_{i})
\qquad
\mu_{i,\jump}=\partial _{t} (v_{i})_{\jump} +\partial_{x} (\lambda_{i} v_{i})_{\jump},
\qquad i=1,\dots,N .
\end{equation}
In Lemmas~\ref{L:waveBalRadon},~\ref{L:jumpWaveBalRadon} below we prove that they are indeed Radon measures: they are controlled by using the interaction and interaction-cancellation measures $\mu^{I}$, $\mu^{IC}$, in the $\nu$-front-tracking approximation together with other terms vanishing in the limit.
The negative part of $\mu_{i,\jump}$ however may not be absolutely continuous w.r.t.~$\mu^{IC}$: this is why the statement holds with the measures $\mu^{ICJ}_{i}$, which dominates $\mu^{IC}+|\mu_{i,\jump}|$ for $i=1,\dots ,\dmns$.

We will first prove estimates of the wave measures on piecewise constant $\nu$-approximate front-tracking solutions $\{u^{\nu}\}$; we then obtain the claims by passing to the limit on a suitable subsequence.

\subsection{Wave and jump wave balance measures}
\label{Ss:wavebalmeas}
We consider the distributions
\begin{align*}
\partial _{t} v_{i} +\partial_{x} (\lambda_{i} v_{i}) &=:\mu_{i}^{}
&
\mu^{\nu}_{i}&:=\partial _{t} v_{i}^{\nu} +\partial_{x} (\bar\lambda_{i}^{\nu} v_{i}^{\nu}) ,
\\
\partial _{t} (v_{i})_{\jump} +\partial_{x} (\lambda_{i} v_{i})_{\jump}&=:\mu_{i,\jump}
&
\mu_{i, \jump}^{\nu,(\varepsilon_{0},\varepsilon_{1})}&:=\partial _{t} \Big(v_{i}^{\nu}\llcorner_{\mathcal J^{\nu,i}_{(\varepsilon_{0},\varepsilon_{1})}} \Big)+\partial_{x} \Big(\bar\lambda_{i} v^{\nu}_{i}\llcorner_{\mathcal J^{\nu,i}_{(\varepsilon_{0},\varepsilon_{1})}} \Big)
.
\end{align*}
Notice first that by Corollary~\ref{C:weakConvergences} the distributions in the left column are limits of the ones in the right column along sequences $\nu_{k}$, $(\varepsilon_{0}^{k},\varepsilon_{1}^{k})$, as $k\to\infty$.
This is why it seems natural to us to define them, even though we will apply in the present paper only the `discrete' ones and we are not going to take advantage of the ones on the l.h.s..

For the rest of the paper we will often omit the index $k$: limits will be tacitly taken on subsequences of the one of Corollary~\ref{C:weakConvergences}.
For every $\nu,(\varepsilon_{0},\varepsilon_{1}),i$ the following holds.

\begin{proposition}
\label{P:mu_estimate}
The distributions $\mu^{}_{i}$, $\mu_{i,\jump}$ and $\mu^{\nu}_{i}$, $\mu_{i,\jump}^{ \nu,(\varepsilon_{0},\varepsilon_{1})}$ are finite Radon measures.
\end{proposition}

We prove in this section the proposition above, direct consequence of the more specific statements in Lemmas~\ref{L:waveBalRadon},~\ref{L:jumpWaveBalRadon}.
We call these measures the \emph{$i$-th wave$\backslash$jump wave balance measures} respectively of $u$ and $u^{\nu}$.
We introduce them with the aim of managing finer balances for the variation of waves, distinguishing different families and the part of their variation only due to shocks.
This is technically more difficult and it is postponed to a future work: in the next section we give only rough balances on strips, for $\nu$-front-tracking approximations, that will be applied in order to derive~\eqref{E:decayEst}.

Before stating the lemmas, let us illustrate in the scalar case what the two measures reduce to.

\begin{example}
In the scalar case $\dmns=1$, the speed $\lambda(u)$ is just $f'(u)$, genuine nonlinearity reads like $f''(u) \geq k>0$, the parameterization choice is
\[
\tilde l
=
\begin{cases}
f''(u)
&\text{at continuity points,}
\\
\frac{f'(u^{+})-f'(u^{-})}{u^{+}-u^{-}}
&\text{at jump points.}
\end{cases}
\]
The wave measure $v=\tilde l u_{x}$ is then
\[
v=f''(u)(\Difxa u+ \Difxc u )+ \frac{f'(u^{+})-f'(u^{-})}{u^{+}-u^{-}}\Difxj u=\big( f'(u)\big)_{x},
\tag{$\circ$}
\]
where the last equality holds by Volpert chain rule  (Th.~3.99 of~\cite{AFP}).

In the case of a smooth solution one can directly compute that the two measures vanish:
\begin{align*}
\mu&= \partial _{t} v+\partial_{x} (\lambda v) = \partial _{t}( f''(u)u_{x})+\partial_{x} (f'(u) f''(u)u_{x}) 
\\
&= \partial _{t}( \partial_{x}f'(u))+\partial_{x} (f'(u) f''(u)u_{x}) 
= \partial _{x}( \partial_{t}f'(u))+\partial_{x} (f'(u) f''(u)u_{x})
\\
&= \partial_{x}\big[ f''(u)\big( u_{t}+f'(u)u_{x}\big)\big] = 0
.
\end{align*}
One can see that the measure $\mu$, if defined on the solution, vanishes also in the $\BV$ case:
\begin{align*}
\int\varphi\mu&= 
-\int\varphi_{t} \big(\tilde l u_{x} \big) -\int \varphi_{x}  \lambda\big(\tilde l u_{x}\big) 
\stackrel{(\star)}{=}
-\int\varphi_{t} \big(\tilde l u_{x} \big) +\int \varphi_{x} \big(\tilde l u_{t}\big) 
\\
&\stackrel{(\circ)}{=} 
-\int\varphi_{t} \big(f'(u)  \big)_{x} +\int \varphi_{x} \big(\tilde l u_{t}\big) 
= 
-\int\varphi_{x} \big(f'(u)  \big)_{t} +\int \varphi_{x} \big(\tilde l u_{t}\big) 
\\
&= 
\int \varphi_{x}\big[-\big(f'(u)  \big)_{t} + \big(\tilde l u_{t}\big)\big] 
\stackrel{(\bullet)}{=}0
,
\end{align*}
where we applied repeatedly Volpert chain rule (for example at $(\bullet)$) and at $(\star)$ we applied the conservation law $-u_{t}=(f(u))_{x}=\lambda u_{x}$.

The measure $\mu_{\jump}$ defined on the solution, in the case of $\BV$ regularity, can be similarly computed as
\[
-\int \varphi\mu_{\jump}
\stackrel{(\circ)}{=} 
\int \big(\varphi_{t}+\lambda\varphi_{x}\big)\lambda(u^{+},u^{-})\Difxj u
=
\sum_{k\in\N}
\int \frac{d}{dt} \Big[\varphi(t,\gamma_{k}(t)) \Big] 
	\Big(\lambda^{+}-\lambda^{-} \Big)(t,\gamma_{k}(t)) dt
,
\]
where we denoted by $\{\gamma_{k}\}_{k\in\N}$ Lipschitz curves covering the jump set of $u_{x}$.
We now make some heuristics. Knowing that $\mu$ is a Radon measure, one obtains that the strength $\lambda^{+}(t,\gamma_{k}(t))-\lambda^{-}(t,\gamma_{k}(t))=\int_{u^{-}(t,\gamma_{k}(t))}^{u^{+}(t,\gamma_{k}(t))}f''(s)ds $ of the jump $\gamma_{k}$ is a function of bounded variation on time intervals where $\gamma_{k} $ is separated from the other curves.
In that case, if one had some regularity of $u^{\pm}(t,\gamma_{k}(t))$ one could derive
\[
\begin{split}
\int \varphi\mu_{\jump}
=
\sum_{k\in\N}
\bigg\{\int \varphi(t,\gamma_{k}(t))
	\big[(f''(u^{+})u^{+}_{t}-f''(u^{-})u^{-}_{t})-\lambda (f''(u^{+})u^{+}_{x}-f''(u^{+})u^{-}_{x}))\big](t,\gamma_{k}(t)) dt\
\\
+ \big\{\varphi \big[f'(u^{+})-f'(u^{-})\big]\big\}|_{(t^{0}_{k},\gamma_{k}(t_{k}^{0}))}- \big\{\varphi \big[f'(u^{+})-f'(u^{-})\big]\big\}|_{(t^{1}_{k},\gamma_{k}(t_{k}^{1}))}
\bigg\}.
\end{split}
\]
\end{example}

We now show that $\mu^{\nu}_{i}$ is a measure concentrated on interaction points, and mainly on interactions between physical waves, where it is controlled by the interaction measure $\mu^{I}$. At the interaction point $(t,x)$, remembering that only one interaction may take place at one time,
\[
v_{i}^{\nu}(t)\big(\R\big)-\lim_{\epsilon\downarrow 0}v^{\nu}_{i}(t-\epsilon)\big(\R\big)= \mu^{\nu}_{i} \big(\{(t,x)\}\big).
\]

\begin{lemma}
\label{L:waveBalRadon}
The distribution $\mu^{\nu}_{i}$ is a Radon measure satisfying $|\mu^{\nu}_{i}|\leq\OO\mu^{I}_{\nu} +\rho^\nu$ for a purely atomic measure $|\rho^\nu|\leq \OO\varepsilon_{\nu}$ concentrated on interaction points involving non-physical waves.
In the $w^{*}$-limit in $\nu$ then one has $|\mu^{}_{i}|\leq \OO\mu^I $.
\end{lemma}

\begin{proof}
We proceed by direct computation of the distribution, that will be a measure concentrated at interaction points.
Then we estimate the value at each point.
\paragr{1) Computing $\mu_{i}^{\nu} $.}
Fix an index $\nu$ and let $\{\ell_m\}_{m=1}^{L_{\nu}}$ be time-parametrized curves whose graphs are respectively the discontinuity segments of $u^\nu$.
Given any test function $\varphi\in C^{\infty}_{\rc}((0,+\infty)\times\R)$ one has
\begin{align*}
-\langle \varphi, \mu_{i}^{\nu}\rangle
&=\iint \big(\varphi_{t}+ \bar\lambda^{}_{i} \varphi_{x})  v^{\nu}_{i}
=\sum_{m=1}^{L_{\nu}}\int_{ \tau^{-}_{m} }^{\tau^{+}_{m}}
\big[ \big(\varphi_{t}+ \bar\lambda^{}_{i} \varphi_{x} \big)  \tilde l_{i}^{} \cdot \big(u^{+}-u^{-}\big) \big]_{x= \ell_m(t)} dt 
\\
&=\sum_{m=1}^{L_{\nu}}\int_{ \tau^{-}_{m} }^{\tau^{+}_{m}} \frac{d}{dt} \Big(\varphi(t, \ell_m(t)) \Big)  \tilde l_{i}^{} \cdot \big(u^{+}-u^{-}\big) dt 
\\
&= \sum_{m=1}^{L_{\nu}} \Big( \varphi( \tau^{+}_{m}, \ell_m( \tau^{+}_{m}))-\varphi(\tau^{-}_{m}, \ell_m( \tau^{-}_{m})) \Big)  \tilde l_{i}^{} \cdot \big(u^{+}-u^{-}\big) .
\end{align*}
As noticed in~\eqref{E:sigmalu}, if $u^+=\Psi_{i}(\sigma_{i})u^-$ by construction $ \tilde l_{j}^{} \cdot \big(u^{+}-u^{-}\big) =0$ for $j\neq i$, and for $j=i$ it is the strength of the wave, that we denote again by $\sigma_{i}$ (with an abuse of notation in case of genuinely non-linear shocks).
The non-vanishing terms in the summation above therefore are the ones of discontinuities of the $i$-th family and non-physical fronts.
If $\{(\tau_{k},z_{k})\}_{k}$ is the collection of `nodes' of $i$-th fronts of $u^\nu$ for $t>0$, the computation above yields
\begin{subequations}
\label{EG:muinu}
\begin{equation}
\mu_{i}^{\nu}
= \sum_{k} p_k \delta_{(\tau_{k},z_{k})} +\rho^\nu,
\end{equation}
where, denoting at each node by $\sigma_i',\sigma_i''$ the $i$-th component of the incoming strengths, $\sigma_i$ the outgoing,
\begin{equation}
p_k=\sigma_i - \sigma'_i - \sigma''_i
\end{equation}
\end{subequations}
while $\rho^\nu$ is concentrated on the nodes of non-physical waves, with
\[
\rho^\nu \big(\{P\} \big)= \tilde l_{i}^{+} \cdot\sigma -\tilde l_{i}^{-} \cdot \sigma',
\]
$\sigma=u^{+}-u^{-}$,  $\sigma'=u'^{+}-u'^{-}$ being the outgoing$\backslash$incoming strengths of the non-physical wave at $P$, $\tilde l_{i}^{\pm} $ the outgoing$\backslash$incoming values of $\tilde l_{i} $ on the non-physical front.
Notice that $\sigma'$ or $\sigma''$ in the above expression may also vanish, for example in the last $\sigma'=0$ when a new non-physical wave originates.

\paragr{2) Estimates.}
In order to bound $\mu^{\nu}_{i}$ one considers the Glimm functional of~\cite{Bressan}
\[
\mathbf\Upsilon ^\nu 
=\sum_{\alpha} |\sigma_{\alpha}|
+ C_{0} \sum_{(\alpha,\beta)} |\sigma_{\alpha}\sigma_{\beta}|
\qquad
C_{0}\gg 1,
\]
where $\alpha,\beta$ index the discontinuities at time $t$, with corresponding strengths $\sigma_{\alpha}$, $\sigma_{\beta}$, and the second summation ranges over couples $(\alpha,\beta)$ whose discontinuity points $x_{\alpha}<x_{\beta}$ either belong to families $k_{\alpha}>k_{\beta}\in\{1,\dots,\dmns+1\}$ or belong to a same family but at least one is a genuinely non-linear shock.
Let $\Upsilon ^\nu $ be the same functional, but neglecting in both the summations all terms involving strengths of non-physical waves.
Since the total strength of non-physical waves at fixed time is controlled by $\varepsilon_{\nu}$ by construction of $\nu$-approximate front-tracking solution, then $|\mathbf\Upsilon ^\nu -\Upsilon^\nu |\leq \OO\varepsilon_{\nu}$.

By choosing  $C_{0}$ big enough and the smallness of the total variation, by those estimates in Lemma~7.2, Page 133, of~\cite{Bressan} and Pages 137, 138, at each interaction at time $\tau$
\[
|\mu_i^\nu|\big(\{Q\}\big)\leq \OO |\sigma'\sigma''|
\leq \OO \big\{\mathbf\Upsilon ^\nu(\tau_{k}^-) - \mathbf\Upsilon ^\nu(\tau_{k}^+)\big\}.
\]
If $\mathfrak T^{p}$ is the set of interaction times involving only physical waves, by definition of the interaction measure
\[
{|\mu_i^\nu|}\llcorner_{\mathfrak T^{p}\times \R}\leq \OO {\mu^I_{\nu} }\llcorner_{\mathfrak T^{p}\times \R},
\qquad
 \mu^I_{\nu}(\R^{+}\times\R)\leq \mathbf\Upsilon^{\nu}(0) .
\]
For the set $\mathfrak T^{n}$ of interaction times involving some non-physical front, instead one has the bound
\begin{align*}
|\rho ^\nu |
&\leq 
\sum_{\tau\in \mathfrak T^{n}} \OO \big\{\mathbf\Upsilon ^\nu(\tau^-) - \mathbf\Upsilon ^\nu(\tau^+)\big\}
\\
&\stackrel{(\circ)}{=}
\sum_{\tau\in \mathfrak T^{n}} \OO \big\{\mathbf\Upsilon ^\nu(\tau^-) - \mathbf\Upsilon ^\nu(\tau^+)- [\Upsilon ^\nu(\tau_{}^-) - \Upsilon ^\nu(\tau_{}^+)]\big\}
\\
&\stackrel{(\star)}{\leq}
\sum_{\tau\in \mathfrak T^{p}\cup \mathfrak T^{n}} \OO \big\{\mathbf\Upsilon ^\nu(\tau_{}^-) - \mathbf\Upsilon ^\nu(\tau_{}^+) 
	- [\Upsilon ^\nu(\tau_{}^-) - \Upsilon ^\nu(\tau_{}^+)]\big\}
+ \OO \varepsilon_{\nu} |\mu^{I}_{\nu} |
\\
&= 
\OO \big\{[\mathbf\Upsilon ^\nu(0) -\Upsilon ^\nu(0) ]- [\mathbf\Upsilon ^\nu(\infty)-\Upsilon ^\nu(\infty) ]\big\}
+ \OO \varepsilon_{\nu} |\mu^{I}_{\nu} |
\leq
\OO \varepsilon_{\nu}
,
\end{align*}
because at interactions involving also non-physical waves $\Upsilon^{\nu}$ does not vary ($\circ$), while at interactions of two physical waves $\sigma'$, $\sigma''$ e.g.~of the $i$-th family most of the terms cancel ($\star$):
\[
\mathbf\Upsilon ^\nu(\tau_{}^-) - \mathbf\Upsilon ^\nu(\tau_{}^+) 
	- [\Upsilon ^\nu(\tau_{}^-) - \Upsilon ^\nu(\tau_{}^+)] 
\leq 
\OO \bigg( \sum_{k\neq i}|\sigma^{k}| +|\sigma_{}'+\sigma_{}''-\sigma_{}| \bigg)\sum_{\beta\text{ non physical}} |\sigma_{\beta}|.
\]

One can gain the estimates on $\mu_{i}$ in the $\nu$-limit because $\mathbf\Upsilon^{\nu}(0)$ can be bounded uniformly in $\nu$.
\end{proof}

The following lemma deals instead with the jump-wave balance measure, which is again concentrated on interactions.
At the interaction point $(t,x)$ one has
\[
v_{i}^{\nu}(t)\big(\mathcal J^{\nu,i}_{(\varepsilon_{0},\varepsilon_{1})}\big)-\lim_{\epsilon\downarrow 0}v^{\nu}_{i}(t-\epsilon)\big(\mathcal J^{\nu,i}_{(\varepsilon_{0},\varepsilon_{1})}\big)= \mu^{\nu}_{i} \big(\{(t,x)\}\big).
\]
The positive part is absolutely continuous w.r.t.~the interaction-cancellation measure, up to a remainder.
The bound on the negative part is obtained instead considering the history of the interacting waves, and involves the Glimm functional $\Upsilon $.

Notations are illustrated in Figure~\ref{fig:lemma}.
Denote by $\{(t_{h},x_{h})\}_{h}$ the finite set of terminal points of the $i$-th $(\varepsilon_{0},\varepsilon_{1})$-jump set $\mathcal J^{\nu,i}_{(\varepsilon_{0},\varepsilon_{1})}$ of $u^{\nu}$ and let $\{\gamma_{h}\}_{h}$, $\{\sigma_{h}\}_{h}$ denote the relative maximal $(\varepsilon_{0},\varepsilon_{1})$-shock front and the strength $\leq \varepsilon_{0}$ just after becoming below the threshold (if any, otherwise vanishes).
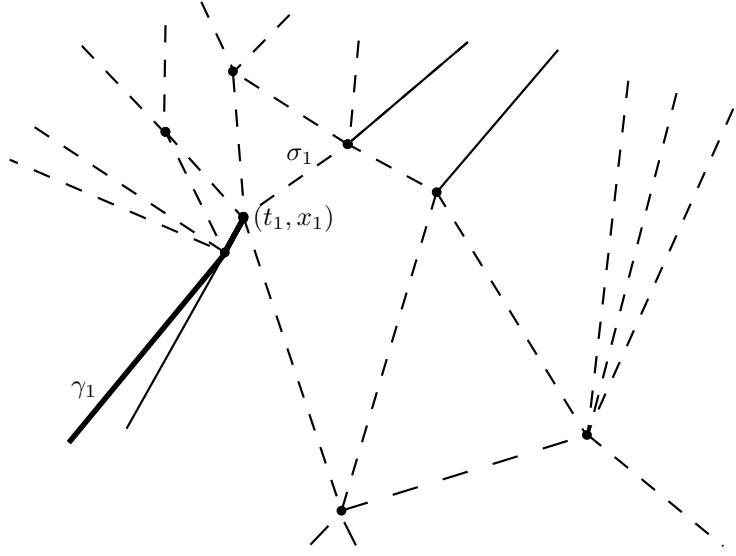
\begin{figure}[ht!] 
   \centering
   \def\svgwidth{.6\columnwidth} 
   \input{lemmaarxiv.pdf_tex}
\caption[Illustration of Lemma~\ref{L:jumpWaveBalRadon}]{Discontinuity lines of $u^{\nu}$ belonging to the $i$-th $(\varepsilon_{0},\varepsilon_{1})$-jump set $\mathcal J^{\nu,i}_{(\varepsilon_{0},\varepsilon_{1})}$.
Dashed lines either are below the threshold (e.g.~$\sigma_{1}$) or belong to a different family $j\neq i$.}
\label{fig:lemma}
\end{figure}
\begin{lemma}
\label{L:jumpWaveBalRadon}
The positive part of $\mu_{i,\jump}^{\nu,(\varepsilon_{0},\varepsilon_{1})}$ satisfies the following inequality
\[
\mu_{i,\jump}^{\nu,(\varepsilon_{0},\varepsilon_{1})}
\leq 
\OO \mu_{\nu}^{IC}-\sum_{h} \sigma_{h}\delta_{(t_{h}, x_{h})}
\qquad
|\sigma_{h}|\leq \OO\frac{\varepsilon_{0}}{\varepsilon_{1}-\varepsilon_{0}}\mu_{\nu}^{IC}(\gamma_{h}) .
\]
Also the negative part of $\mu^{\nu,(\varepsilon_{0},\varepsilon_{1})}_{i,\jump}$ is concentrated on the set of nodes of $\mathcal J^{\nu,i}_{(\varepsilon_{0},\varepsilon_{1})}$ and its mass is uniformly bounded by the Glimm functional $\Upsilon$.
In the $w^{*}$-limit of Corollary~\ref{C:weakConvergences} one gets
\[
-\rho_{}
\leq 
\mu_{i,\jump}
\leq 
\OO \mu_{}^{IC} 
\qquad
\text{$\rho_{}\geq 0$ finite.}
\]
\end{lemma}
\begin{proof}
As before, we proceed by direct computation of the measure.
\paragr{1) Computing $\mu^{\nu,(\varepsilon_{0},\varepsilon_{1})}_{i,\jump} $.}
One has to take into account that some of the shocks of $u^{\nu}$ are not present in the maximal $(\varepsilon_{0},\varepsilon_{1})$-shock fronts, and none of the rarefaction fronts and non-physical waves appears.
Fix an index $\nu$ and let $\{\gamma_m\}_{m=1}^{L_{\nu}}$ be time-parametrized curves whose graphs are respectively the discontinuity \emph{segments} of $u^\nu$ \emph{present} in $\mathcal J^{\nu,i}_{(\varepsilon_{0},\varepsilon_{1})}$.
Given any test function $\varphi\in C^{\infty}_{\rc}((0,+\infty)\times\R)$, and omitting for simplicity most of the indices $\nu,(\varepsilon_{0},\varepsilon_{1}),i $ from the second line below up to the end of the proof, one has
\begin{align*}
-\langle \varphi, \mu^{\nu,(\varepsilon_{0},\varepsilon_{1})}_{i,\jump} \rangle
&=\iint \big(\varphi_{t}+ \bar\lambda^{\nu}_{i} \varphi_{x})  v^{\nu}_{i}\llcorner_{\mathcal J^{\nu,i}_{(\varepsilon_{0},\varepsilon_{1})}} \\
&=\sum_{m}
\int_{ \tau^{-}_{m} }^{\tau^{+}_{m}} \frac{d}{dt} \Big(\varphi(t,\gamma_{m}(t)) \Big)  \tilde l_{}^{} \cdot \big(u^{+}-u^{-}\big) dt 
\\
&= \sum_{m}
\Big( \varphi( \tau^{+}_{m},\gamma_{m}( \tau^{+}_{m}))-\varphi(\tau^{-}_{m}, \gamma( \tau^{-}_{m}) ) \Big)\sigma_{i} .
\end{align*}
If $\{(\tau_{k},z_{k})\}_{k}$ are the nodes in $\mathcal J^{\nu,i}_{(\varepsilon_{0},\varepsilon_{1})}$ for $t>0$, then the computation above can be rewritten as
\begin{subequations}
\label{EG:mujump}
\begin{equation}
\mu^{\nu,(\varepsilon_{0},\varepsilon_{1})}_{i,\jump}=\sum_{k} q_{k}\delta_{(\tau_{k},z_{k})},
\end{equation}
where, denoting at each node by $\sigma',\sigma''$ the $i$-incoming strengths, $\sigma$ the outgoing, one has
\begin{equation}
q_{k}=
\begin{cases}
\sigma-\sigma'-\sigma''	&\text{at triple points of $\mathcal J^{\nu,i}_{(\varepsilon_{0},\varepsilon_{1})}$},
\\
\sigma	&\text{at initial points of maximal fronts of $\mathcal J^{\nu,i}_{(\varepsilon_{0},\varepsilon_{1})}$},
\\
-\sigma'	&\text{when a maximal front of $\mathcal J^{\nu,i}_{(\varepsilon_{0},\varepsilon_{1})}$ ends without merging into another},
\\
\sigma-\sigma'	&\text{otherwise}.
\end{cases}
\end{equation}
\end{subequations}

\paragr{2) Upper bound.}
At triple points of $\mathcal J^{\nu,i}_{(\varepsilon_{0},\varepsilon_{1})}$ one has $|q_{k}|\leq \OO\mu^{I}_{\nu}$ by the interaction estimates in Lemma~7.2 of~\cite{Bressan} and by definition of the interaction measure~\eqref{E:intrInterCanc}.

At initial points of $\mathcal J^{\nu,i}_{(\varepsilon_{0},\varepsilon_{1})}$, by genuine nonlinearity $q_{k}\leq 0$.

At internal nodal points of a front where
\begin{itemize}
\item[-] another shock of the same family, not belonging to any front in $\mathcal J^{\nu,i}_{(\varepsilon_{0},\varepsilon_{1})} $, merges, again $q_{k}\leq 0$;
\item[-] a rarefaction wave front interacts, a cancellation occurs and $|q_{k}|= |\sigma-\sigma'| \leq\mu^{IC}_{\nu}(t_{k},x_{k})$ by definition of the interaction-cacellation measure in~\eqref{E:intrInterCanc};
\item[-] any curve of different family interact, interaction takes place and $|q_{k}|= |\sigma-\sigma'| \leq \OO\mu^{I}_{\nu}(t_{k},x_{k})$.
\end{itemize}

At terminal points, since the shock gets cancelled from $\mathcal J^{\nu,i}_{(\varepsilon_{0},\varepsilon_{1})}$, there must be a cancellation by a rarefaction front or by a wave of different family. If $\gamma_{k}$ is the maximal $(\varepsilon_{0},\varepsilon_{1})$-shock front through the point, by inspection one can see that the strength must increase from a value $\sigma_{0}\leq -\varepsilon_{1}$ at some point up to a value $\leq -\varepsilon_{0}$ on the last segment, which colliding the other wave front becomes $\sigma_{1} $ and exceeds $-\varepsilon_{0}$.
By the interaction estimates, at nodal points of each front the strength may increase at most of the amount of interaction-cancellation at that point. One can then conclude
\[
\varepsilon_{1}-\varepsilon_{0} \leq
|\sigma_{0}|-|\sigma_{1}| \leq \TV^{+}(u, \gamma_{k} ) \leq \OO \mu^{IC}_{\nu}(\gamma_{k}).
\]
This yields the final estimate, at terminal points:
%
\[
q_k
=
-\sigma_1 + (\sigma_1-q_k)
\leq
\frac{\varepsilon_{0}}{\varepsilon_{1}-\varepsilon_{0}}(\varepsilon_{1}-\varepsilon_{0})
+ \OO\mu^{I}_{\nu}(t_{k},x_{k})
\leq
\frac{\varepsilon_0}{\varepsilon_1-\varepsilon_0}
\OO\mu^{IC}_\nu(\gamma_k) + \OO\mu^{I}_{\nu}(t_{k},x_{k}).
\]
Notice that $\sigma_{1}$ is one of the $\sigma_{k}$ in the first statement.

By summing up the different contributions, being
\[
\sum | \sigma_{h}| \leq \OO\frac{\varepsilon_{0}}{\varepsilon_{1}-\varepsilon_{0}}    \mu^{IC}_{\nu}(\R^{+}\times \R), 
\]
one gains that $\mu^{\nu,(\varepsilon_{0},\varepsilon_{1})}_{i,\jump}-\OO\mu_{\nu}^{IC}+ \sum  \sigma_{h} \delta_{(t_{h},x_{h})}$ is a signed measure, and thus $\mu^{\nu,(\varepsilon_{0},\varepsilon_{1})}_{i,\jump} $ is a Radon measure.
We got an upper bound, uniform w.r.t~$\nu$, of its positive mass, and we also got that the limit in $\nu$ in the sense of distributions, i.e.~$\mu_{i,\jump}$, is a Radon measure.
Notice that $\sum |\sigma_{h}|$ by the estimate above vanishes in the limit of Corollary~\ref{C:weakConvergences}, as we have chosen $\varepsilon^{k}_{1}\geq 2^{k}\varepsilon_{0}^{k}$.

\paragr{3) Lower bound.}
We derive now a lower bound uniform w.r.t~$\nu$ of the negative mass of $\mu^{\nu,(\varepsilon_{0},\varepsilon_{1})}_{i,\jump} $, suitable to estimate that the negative part of the limit $\mu_{i,\jump}$ is actually a finite measure.

Consider the nonnegative measure $\bar\mu= -\mu_{i,\jump}^{\nu,(\varepsilon_{0},\varepsilon_{1})} +\OO\mu_{\nu}^{IC}+ \sum | \sigma'_{h}| \delta_{(t_{h},x_{h})}$.
Consider for $\alpha>0$ the Lipschitz test function $\varphi_{\alpha}(t)= \chi_{[0,T+\alpha]}{(t)} -(t-T)/\alpha\chi_{[T, T+\alpha]}(t)$: being the time marginal $ v^{\nu}_{i,\jump}(t)$ of $v^{\nu}_{i,\jump}$ absolutely continuous and locally finite one has
\begin{align*}
\bar \mu([0,T]\times\R)
&\leq
\int \varphi_{\alpha} d\bar \mu
=
-\int \varphi_{\alpha} \mu_{i,\jump}^{\nu,(\varepsilon_{0},\varepsilon_{1})} 
+ \OO\int \varphi_{\alpha} d\mu_{\nu}^{IC}+ \sum_{h} | \sigma'_{h}| \varphi_{\alpha}{(t_{h})}
\\
&\leq \iint((\varphi_{\alpha})_{t}+\bar\lambda_{i}^{\nu} (\phi_{\alpha})_{x}) d\big[v^{\nu}_{i,\jump}(t)\big]dt
+ \big[v^{\nu}_{i,\jump}(0)\big](\R)
+\OO \mu^{IC}_{\nu}([0,T+\alpha]\times\R)
\\
&= - \frac{1}{\alpha}\int_{T_{}}^{T_{}+\alpha}\big[v^{\nu}_{i,\jump}(t)\big](\R)dt
+ \big[v^{\nu}_{i,\jump}(0)\big](\R)
+\OO \mu_{\nu}^{IC}([0,T+\alpha]\times\R)
\end{align*}
By the uniform bound on $|v^{\nu}_{i,\jump}(t)|$, as $\alpha\downarrow0$ one obtains
\[
\bar \mu([0,T]\times\R)\leq
\Delta_{0,t}\TV(v^{\nu}_{i,\jump})+\OO \mu_{\nu}^{IC}([0,T+\alpha]\times\R)
\leq\OO\Upsilon^{\nu}(0).
\qedhere
\]
\end{proof}

\subsection{Balances on characteristic regions}
\label{Ss:balchar}
We already discussed the distinction between the part of $u_{x}^{\nu}$ which approximates the jump part of $u_{x}$ and what remains: now we denote shortly
\begin{align*}
v^{\nu}_{i,\jump}=v_{i}^{\nu}\llcorner_{\mathcal J^{\nu,i}_{(\varepsilon_{0},\varepsilon_{1})}} 
&&
\mu^\nu_{i,\jump}=\mu_{i,\jump}^{\nu,(\varepsilon_{0},\varepsilon_{1})} 
\\
v^{\nu}_{i,\cont}=v_{i}^{\nu}-v^{\nu}_{i,\jump}
&&
\mu^\nu_{i, \cont}=\mu_{i}^{\nu}-\mu^\nu_{i,\jump}
.
\end{align*}
In the proof of Lemmas~\ref{L:waveBalRadon},~\ref{L:jumpWaveBalRadon} we derived the balances
\[
v^{\nu}_{i}(t)-v^{\nu}_{i}(t^{-})=\mu^\nu_{i}(\{t\}\times\R),
\qquad
v^{\nu}_{i,\jump}(t)-v^{\nu}_{i,\jump}(t^{-})=\mu^\nu_{i,\jump}(\{t\}\times\R).
\]
We estimate in this section balances of these measures in regions bounded by generalized $\bar\imath $-th characteristics, based on the assumption of genuine non-linearity of the $\bar\imath $-th characteristic field.
A relevant measure will be the approximated interaction-cancellation-jump balance measure
\begin{equation}
\label{E:muICJinu}
\mu^{ICJ}_{\bar\imath,\nu} = \mu^{IC} + |\mu_{\bar\imath,\jump}^{\nu}| .
\end{equation}
{The measure $\mu^{ICJ}_{\bar\imath}$ is defined as a $w^{*}$-limit of the measures $\mu^{ICJ}_{\bar\imath,\nu}$.}

We recall that a \emph{(generalized) $i$-th characteristic} is an absolutely continuous curve $x(t)$ satisfying for a.e.~$t$ the differential inclusion
\[
\dot x(t)\in[\lambda_{i}(u(t, x(t)^{+})), \lambda_{i}(u(t, x(t)^{-}))] .
\]
Due to the presence of discontinuities, if $u$ is a semigroup solution of a system of conservation laws or a $\nu$-approximate front-tracking solution there are several backward and forward characteristics, and they may collapse. 
They are polygonal lines whose direction changes at interaction points or hitting discontinuities, which are their nodes.

One can select the \emph{minimal} $i$-th characteristics of the piecewise constant approximations starting at $t_{0}$ from any point $\bar x$, which means
\[
 x(t;t_{0},x_{0})
 =\min\bigg\{\begin{array}{c}x(t):\ x(t= t_{0})= x_{0},\ \dot x(t)\in[\lambda_{i}(u(t, x^{+})), \lambda_{i}(u(t, x^{-}))]
 \end{array} \bigg\}.
\]
Analogously, one can select the \emph{maximal} one.

For the rest of the paper we will select for simplicity minimal generalized characteristics. We stress however that the same statements hold if one selects any other family $y^{\nu}(t;t_{0},x)$ of generalized characteristics of $u^{\nu}$ which, as the minimal and maximal ones, have the semigroup property
\[
y^{\nu}(t_{0};t_{0}, x_{0})= x_{0};
\qquad
y^{\nu}(t+\tau;t_{0}, x_{0})= y^{\nu}(\tau; t, y^{\nu}(t,t_{0},x_{0})).
\]
\begin{notation}Given an interval $I=[a,b]$, we define the region $A^{t_{0},\tau}_{[a,b]}$ bounded by the minimal $i$-characteristics $a(t),b(t)$ starting at $t_{0}$ respectively from $a,b$, and its time-section $I(t)$, as
\begin{equation}
\label{E:At0tau}
A^{t_{0},\tau}_{[a,b]}= \big\{(t,x):\  t_{0}< t \leq t_{0}+\tau,\ a(t) \leq x \leq b(t) \big\},
\qquad
I(t)=[a(t),b(t)] .
\end{equation}
Given $J=I_{1}\cup \dots\cup I_{M}$ the union of any disjoint closed intervals $I_{1},\dots, I_{M}$, $M\in\N$, we denote by $A^{t_{0},\tau}_{J}$  the union of the regions $A^{t_{0},\tau}_{I_{h}}$.
Similarly, $J(t)=A^{t_{0},\tau}_{J}\cap \{t\}\times \R=I_{1}(t)\cup\dots \cup I_{M}(t)$.
\end{notation}
\begin{lemma}[Approximate wave balances]
\label{L:waveBalances}
Assume~\eqref{E:genNonl} and consider regions bounded by $\bar\imath$-characteristics as above.
Then for any $\tau,t_{0}\geq 0$ and $t=t_{0}+\tau_{0}$
\begin{align*}
v_{\bar\imath }^{\nu}(t) \big(J(t)\big)-v_{\bar\imath}^{\nu}(t_{0})\big(J\big)
&\leq 
\OO \big(\mu^{IC}_{\nu}\big(A^{t_{0},\tau}_{J}\big) +\varepsilon_{\nu}\big) ,
\\
v^{\nu}_{\bar\imath,\jump}(t) \big(J(t)\big)-v^{\nu}_{\bar\imath,\jump}(t_{0})\big(J\big)
&\leq 
\mu^\nu_{\bar\imath,\jump}\big(A^{t_{0},\tau}_{J}\big) ,
\\
v_{\bar\imath ,\cont}^{\nu}(t) \big(J(t)\big)-v_{\bar\imath ,\cont}^{\nu}(t_{0})\big(J\big)
&\leq 
\OO \big(\mu^{ICJ}_{\bar\imath,\nu}\big(A^{t_{0},\tau}_{J}\big) +\varepsilon_{\nu}\big) .
\end{align*}
\end{lemma}
The lemma shows that the variation of both the `continuous' and `jump' part of the wave measures on characteristic regions is controlled by finite, purely atomic measures, concentrated on interaction points. While for points in the interior of the region a perfect balance holds, with the suitable measures, in the proof one defines `fluxes through the boundaries' in order to take into account the variation due to waves that enter$\backslash$exit the region.
Then the fluxes are controlled by the assumption of genuine non-linearity and the interaction-cancellation measure, by the choice of the speed of the rarefaction as a mean of the values $\lambda(u^{\pm})$.
\begin{proof}
Fix for simplicity of notations $t_{0}=0$, $A^{\tau}_{-}=A^{0,\tau}_{-}$.
Given a closed interval $[a,b]$ we briefly denote
\begin{subequations}
\label{EG:balClosedInt}
\begin{equation}
v_{\bar \imath}^{\nu}(t):= v_{\bar \imath}^{\nu}(t)\big([a(t),b(t)]\big)
.
\end{equation}
We prove before the first balance, which is most of the work, and then we explain the others.

\paragr{0) Contribution of non-physical waves.}
The total strength of non-physical fronts at each finite time is controlled by $\varepsilon_{\nu}$, as well as the mass $|\rho^{\nu}|$ of $\mu ^{\nu}$ due to interactions involving non-physical waves (Lemma~\ref{L:waveBalRadon}).
Since thus non-physical interactions and waves are not relevant, they may be neglected by removing from both $\mu^{\nu}$, $v^{\nu}$ the relative terms: denote by $\hat \mu^{\nu} = \mu ^{\nu}-\rho^{\nu} $, $\hat v^{\nu}=v^{\nu}\llcorner_{\{\text{physical waves}\}}$ this simplified measures.
It suffices to obtain the claim for these measures, then the other follows.

\paragr{1) Instantaneous estimate.}
The region $A^{\tau}_{[a,b]}$ is bounded by generalized $\bar\imath$-characteristics: then the function $v_{\bar \imath}^{\nu}(t)$ may vary only at times $\hat t$ where an interaction takes place in $[a(\hat t), b(\hat t)]$, an $i$-rarefaction wave leaves the region $A^{\tau}_{[a,b]}$ or an $i$-shock enters.
The following cases may occur: if physical waves in $A^{\tau}_{[a,b]}$ change at $(\hat t,\hat x)$, then one can define $\Phi_{\bar \imath,[a,b]}^{\nu,\out}$ by
\[
v^{\nu}_{\bar \imath}(\hat t^{+})-v^{\nu}_{\bar \imath}(\hat t^{-})= \mu_{\bar \imath}^\nu\big(\{(\hat t,\hat x)\} \big) +\Phi_{\bar \imath,[a,b]}^{\nu,\out}(\{(\hat t,\hat x)\}).
\]
Formula~\eqref{EG:muinu} and the inspection of all the possible cases, that we perform without taking into account that $a(t),b(t)$ are minimal, yield that $\Phi_{\bar \imath,[a,b]}^{\nu,\out}(\{(\hat t,\hat x)\})$ is different from $0$ only if $\hat x$ coincides with $a(\hat t)$ or $b(\hat t)$.
The value can be computed exactly. The most relevant cases are described below, and they correspond to interaction among physical waves of the same family.
The others are analogous.
For example, we do not consider when only part of a rarefaction exits the region: one can see that the value of $\Phi_{\bar \imath,[a,b]}^{\nu,\out}$ is controlled by above and below by two cases we consider (the whole rarefaction exists, or no part of it exits).
In these cases, that we illustrate in Figure~\ref{fig:interazioni}, $\Phi_{\bar \imath,[a,b]}^{\nu,\out}$ is precisely equal to
\newcommand{\larghezzaimmagine}{.15 \linewidth} 
\newcommand{\spazioimmagine}{\hspace{.03\linewidth}}
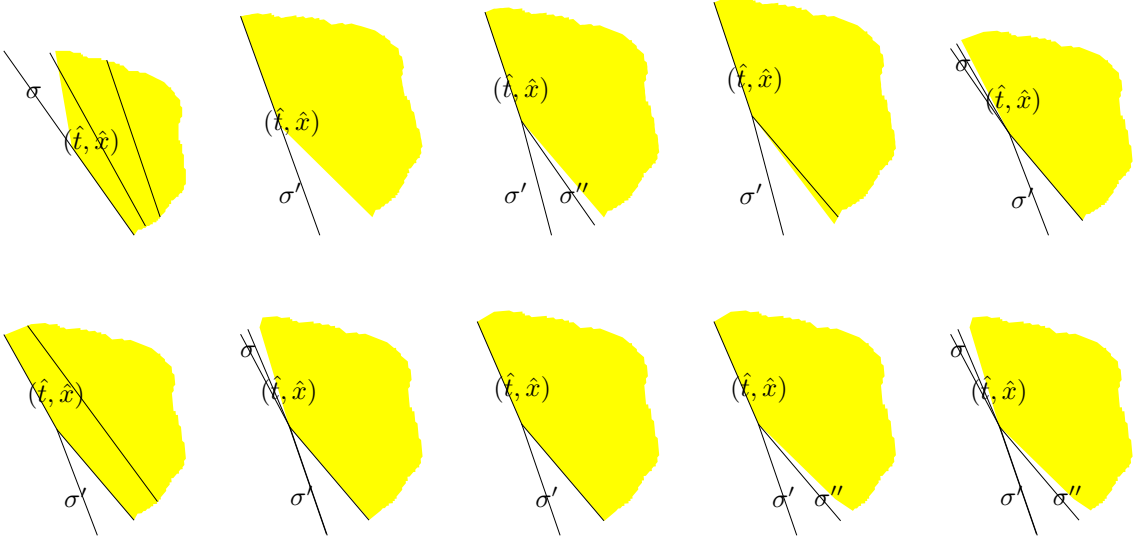
\begin{figure}[ht!]
\centering
   \def\svgwidth{\larghezzaimmagine} 
  \ifx\svgwidth\undefined
    \setlength{\unitlength}{42.08501606pt}
  \else
    \setlength{\unitlength}{\svgwidth}
  \fi
  \global\let\svgwidth\undefined
  \makeatother
  \begin{picture}(1,1.02078347)%
    \put(0,0){\includegraphics[width=\unitlength]{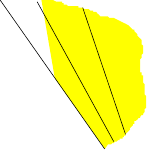}}%
    \put(0.12255952,0.75761011){\color[rgb]{0,0,0}\makebox(0,0)[lb]{\smash{$\sigma$}}}%
    \put(0.32409623,0.47545871){\color[rgb]{0,0,0}\makebox(0,0)[lb]{\smash{$(\hat t,\hat x)$}}}%
  \end{picture}%
\spazioimmagine
   \def\svgwidth{\larghezzaimmagine} 
  \ifx\svgwidth\undefined
    \setlength{\unitlength}{35.04536801pt}
  \else
    \setlength{\unitlength}{\svgwidth}
  \fi
  \global\let\svgwidth\undefined
  \makeatother
  \begin{picture}(1,1.2287654)%
    \put(0,0){\includegraphics[width=\unitlength]{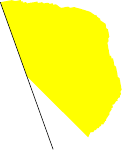}}%
    \put(0.12617551,0.58013578){\color[rgb]{0,0,0}\makebox(0,0)[lb]{\smash{$(\hat t,\hat x)$}}}%
    \put(0.21146593,0.18104099){\color[rgb]{0,0,0}\makebox(0,0)[lb]{\smash{$\sigma'$}}}%
  \end{picture}%
\spazioimmagine
   \def\svgwidth{\larghezzaimmagine} 
  \ifx\svgwidth\undefined
    \setlength{\unitlength}{36.01071429pt}
  \else
    \setlength{\unitlength}{\svgwidth}
  \fi
  \global\let\svgwidth\undefined
  \makeatother
  \begin{picture}(1,1.25895071)%
    \put(0,0){\includegraphics[width=\unitlength]{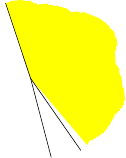}}%
    \put(0.08430031,0.76385996){\color[rgb]{0,0,0}\makebox(0,0)[lb]{\smash{$(\hat t,\hat x)$}}}%
    \put(0.15,0.17990677){\color[rgb]{0,0,0}\makebox(0,0)[lb]{\smash{$\sigma'$}}}%
    \put(0.45403154,0.17990677){\color[rgb]{0,0,0}\makebox(0,0)[lb]{\smash{$\sigma''$}}}%
  \end{picture}%
\spazioimmagine
   \def\svgwidth{\larghezzaimmagine} 
  \ifx\svgwidth\undefined
    \setlength{\unitlength}{35.9107438pt}
  \else
    \setlength{\unitlength}{\svgwidth}
  \fi
  \global\let\svgwidth\undefined
  \makeatother
  \begin{picture}(1,1.31462303)%
    \put(0,0){\includegraphics[width=\unitlength]{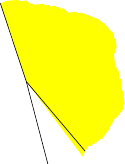}}%
    \put(0.06866197,0.82249379){\color[rgb]{0,0,0}\makebox(0,0)[lb]{\smash{$(\hat t,\hat x)$}}}%
    \put(0.14322701,0.17626346){\color[rgb]{0,0,0}\makebox(0,0)[lb]{\smash{$\sigma'$}}}%
  \end{picture}%
\spazioimmagine
   \def\svgwidth{\larghezzaimmagine} 
  \ifx\svgwidth\undefined
    \setlength{\unitlength}{416.31456103pt}
  \else
    \setlength{\unitlength}{\svgwidth}
  \fi
  \global\let\svgwidth\undefined
  \makeatother
  \begin{picture}(1,1.1308811)%
    \put(0,0){\includegraphics[width=\unitlength]{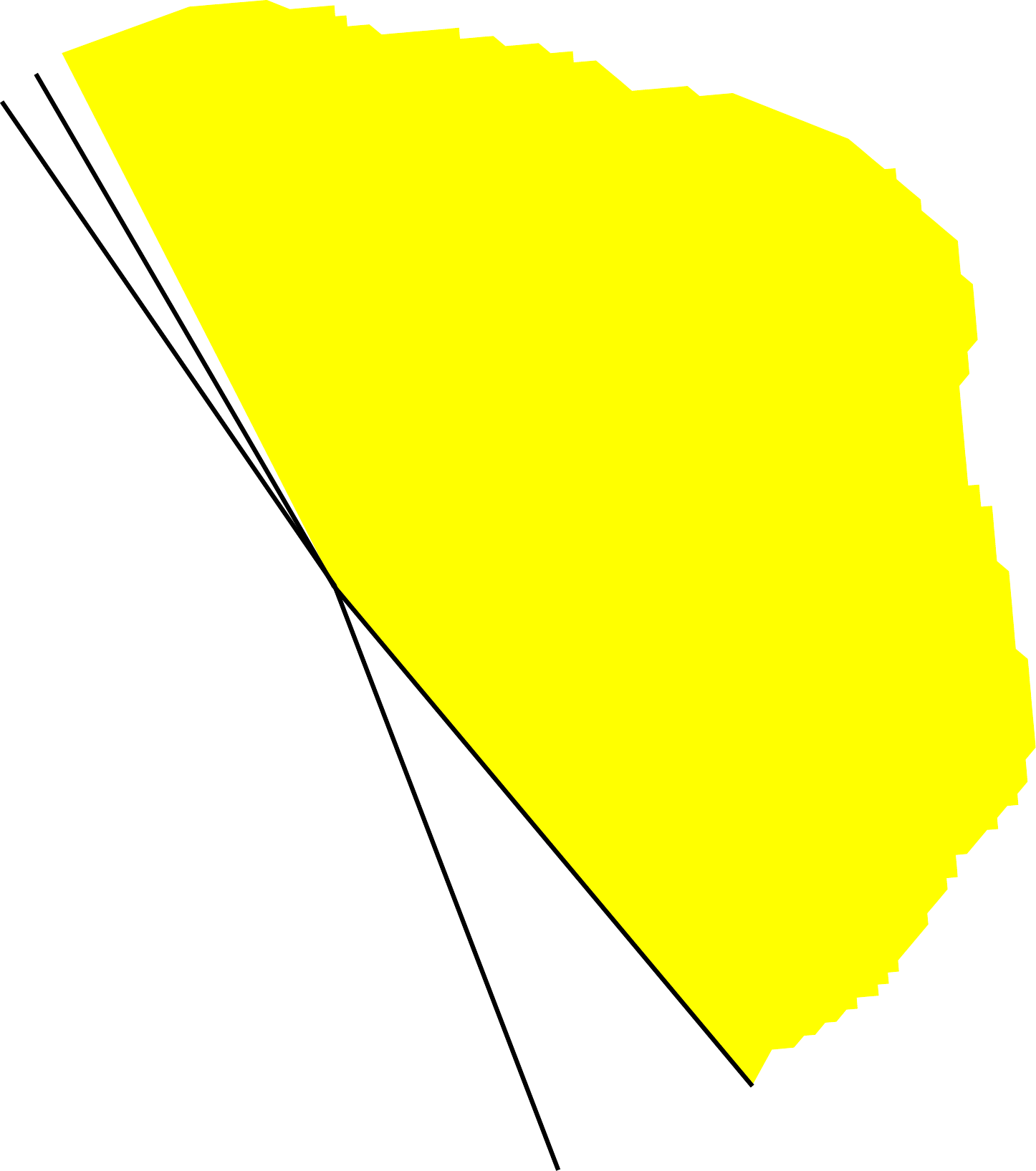}}%
    \put(0.18792679,0.70314584){\color[rgb]{0,0,0}\makebox(0,0)[lb]{\smash{$(\hat t,\hat x)$}}}%
    \put(0.02797691,0.9060545){\color[rgb]{0,0,0}\makebox(0,0)[lb]{\smash{$\sigma$}}}%
    \put(0.33441625,0.15241419){\color[rgb]{0,0,0}\makebox(0,0)[lb]{\smash{$\sigma'$}}}%
  \end{picture}%
 \\
 \vspace{.06\linewidth}
   \def\svgwidth{\larghezzaimmagine} 
  \ifx\svgwidth\undefined
    \setlength{\unitlength}{401.95177221pt}
  \else
    \setlength{\unitlength}{\svgwidth}
  \fi
  \global\let\svgwidth\undefined
  \makeatother
  \begin{picture}(1,1.17408)%
    \put(0,0){\includegraphics[width=\unitlength]{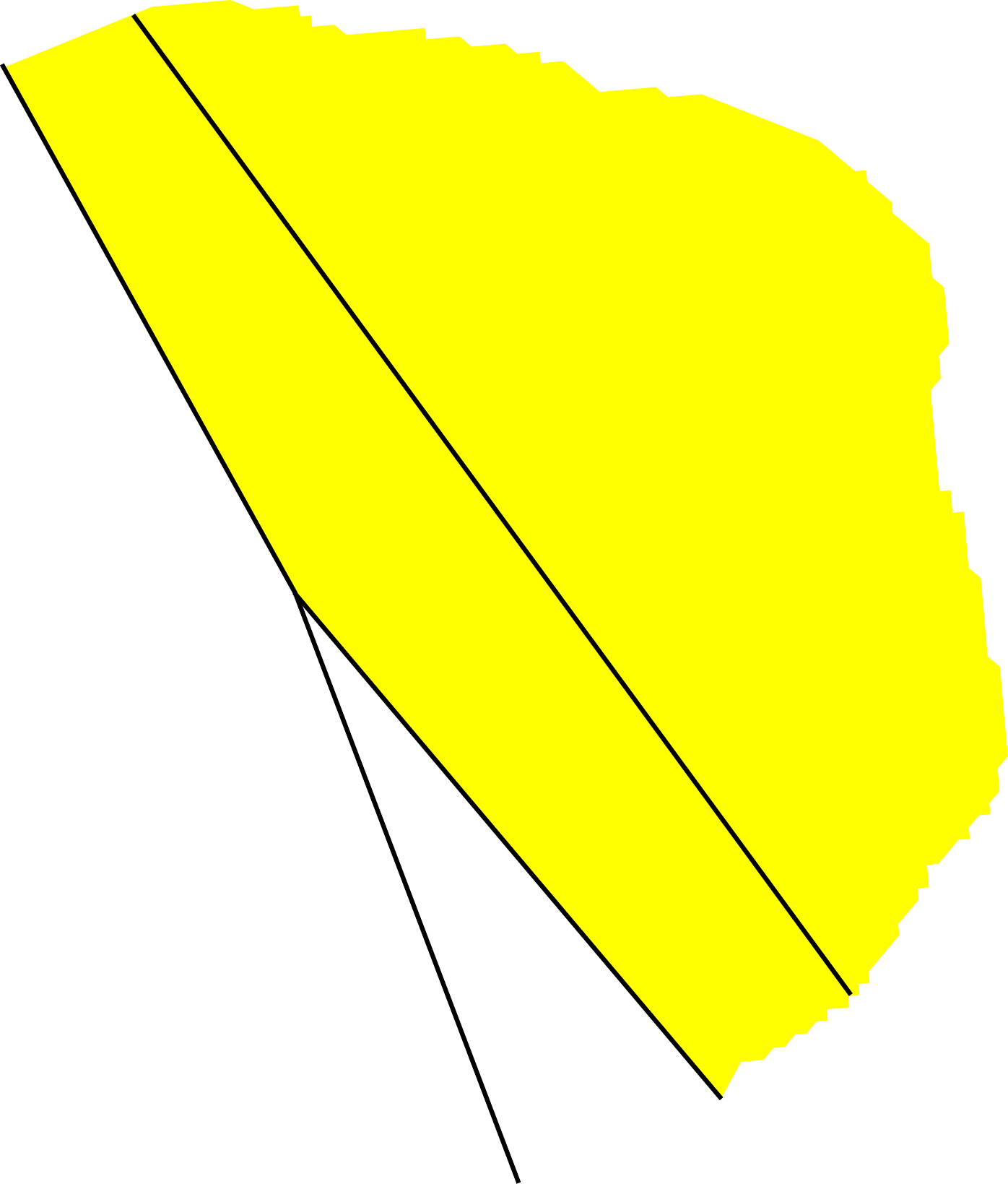}}%
    \put(0.130037,0.73999107){\color[rgb]{0,0,0}\makebox(0,0)[lb]{\smash{$(\hat t,\hat x)$}}}%
    \put(0.33280027,0.15550328){\color[rgb]{0,0,0}\makebox(0,0)[lb]{\smash{$\sigma'$}}}%
  \end{picture}%
\spazioimmagine
   \def\svgwidth{\larghezzaimmagine} 
  \ifx\svgwidth\undefined
    \setlength{\unitlength}{387.67479301pt}
  \else
    \setlength{\unitlength}{\svgwidth}
  \fi
  \global\let\svgwidth\undefined
  \makeatother
  \begin{picture}(1,1.21123979)%
    \put(0,0){\includegraphics[width=\unitlength]{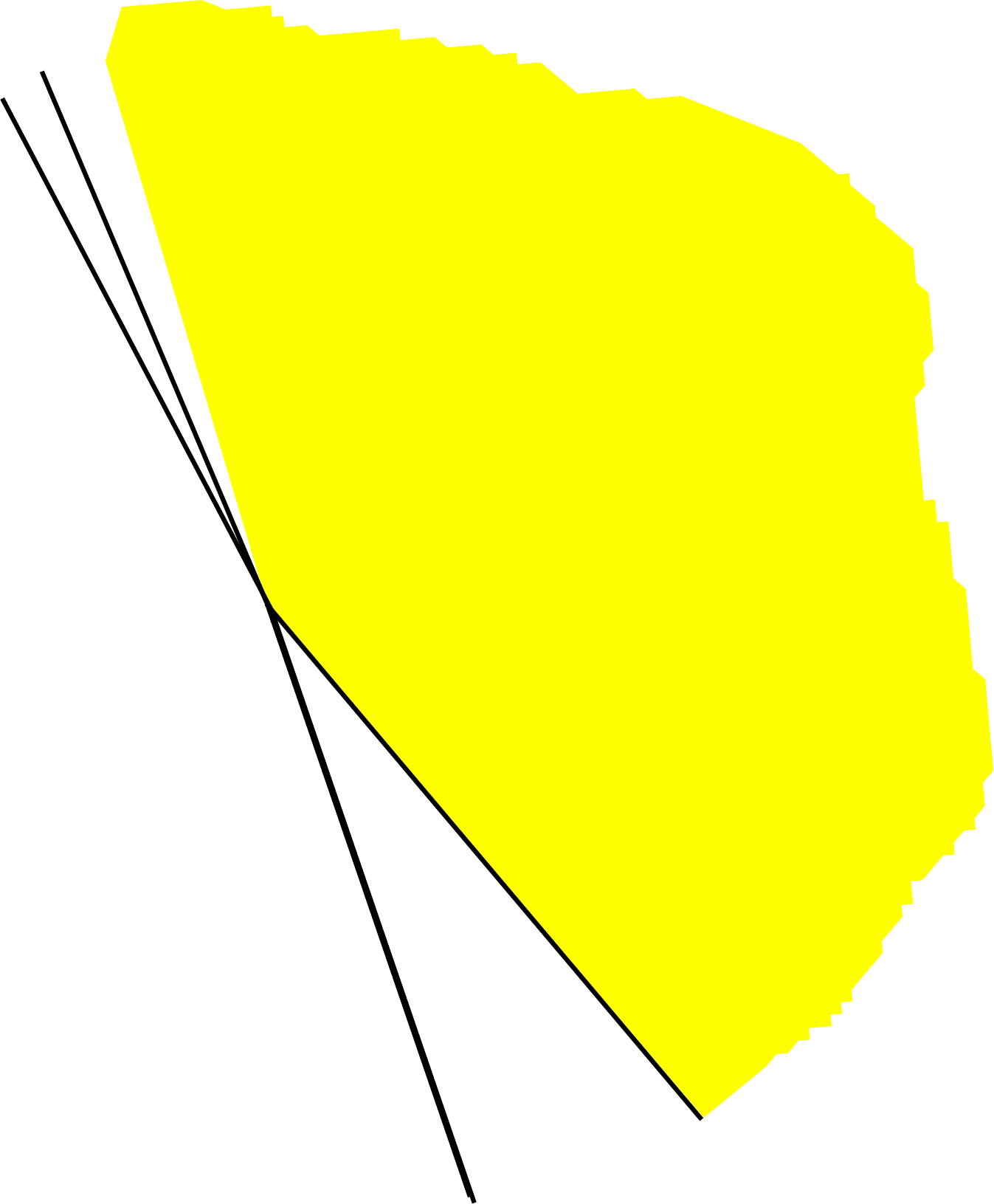}}%
    \put(0.00017299,0.98657857){\color[rgb]{0,0,0}\makebox(0,0)[lb]{\smash{$\sigma$}}}%
    \put(0.10904373,0.75207593){\color[rgb]{0,0,0}\makebox(0,0)[lb]{\smash{$(\hat t,\hat x)$}}}%
    \put(0.27481019,0.16681884){\color[rgb]{0,0,0}\makebox(0,0)[lb]{\smash{$\sigma'$}}}%
  \end{picture}%
\spazioimmagine
   \def\svgwidth{\larghezzaimmagine} 
  \ifx\svgwidth\undefined
    \setlength{\unitlength}{380.3799935pt}
  \else
    \setlength{\unitlength}{\svgwidth}
  \fi
  \global\let\svgwidth\undefined
  \makeatother
  \begin{picture}(1,1.2432072)%
    \put(0,0){\includegraphics[width=\unitlength]{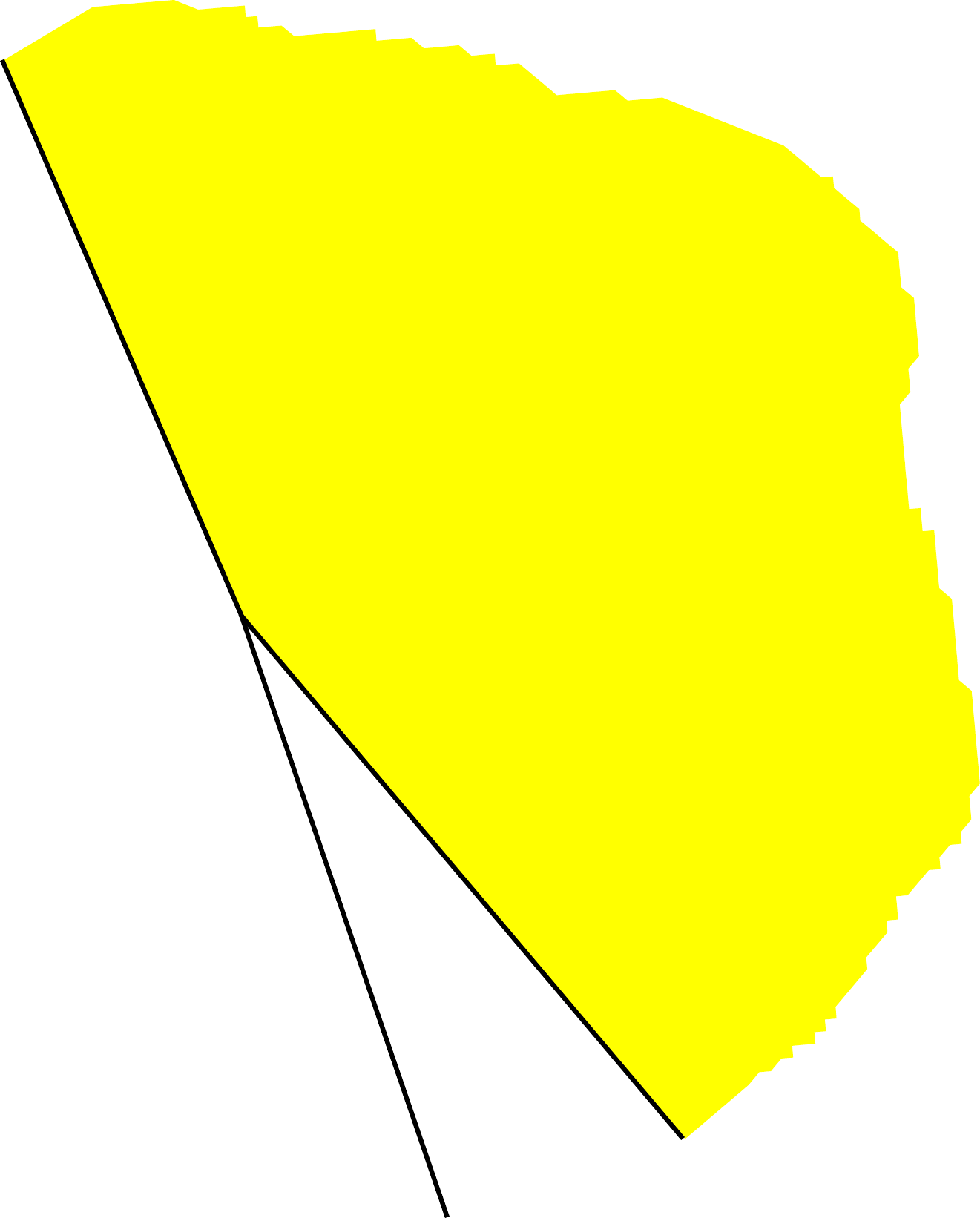}}%
    \put(0.0916563,0.76951334){\color[rgb]{0,0,0}\makebox(0,0)[lb]{\smash{$(\hat t,\hat x)$}}}%
    \put(0.32120106,0.1592447){\color[rgb]{0,0,0}\makebox(0,0)[lb]{\smash{$\sigma'$}}}%
  \end{picture}%
\spazioimmagine
   \def\svgwidth{\larghezzaimmagine} 
  \ifx\svgwidth\undefined
    \setlength{\unitlength}{380.3799935pt}
  \else
    \setlength{\unitlength}{\svgwidth}
  \fi
  \global\let\svgwidth\undefined
  \makeatother
  \begin{picture}(1,1.2432072)%
    \put(0,0){\includegraphics[width=\unitlength]{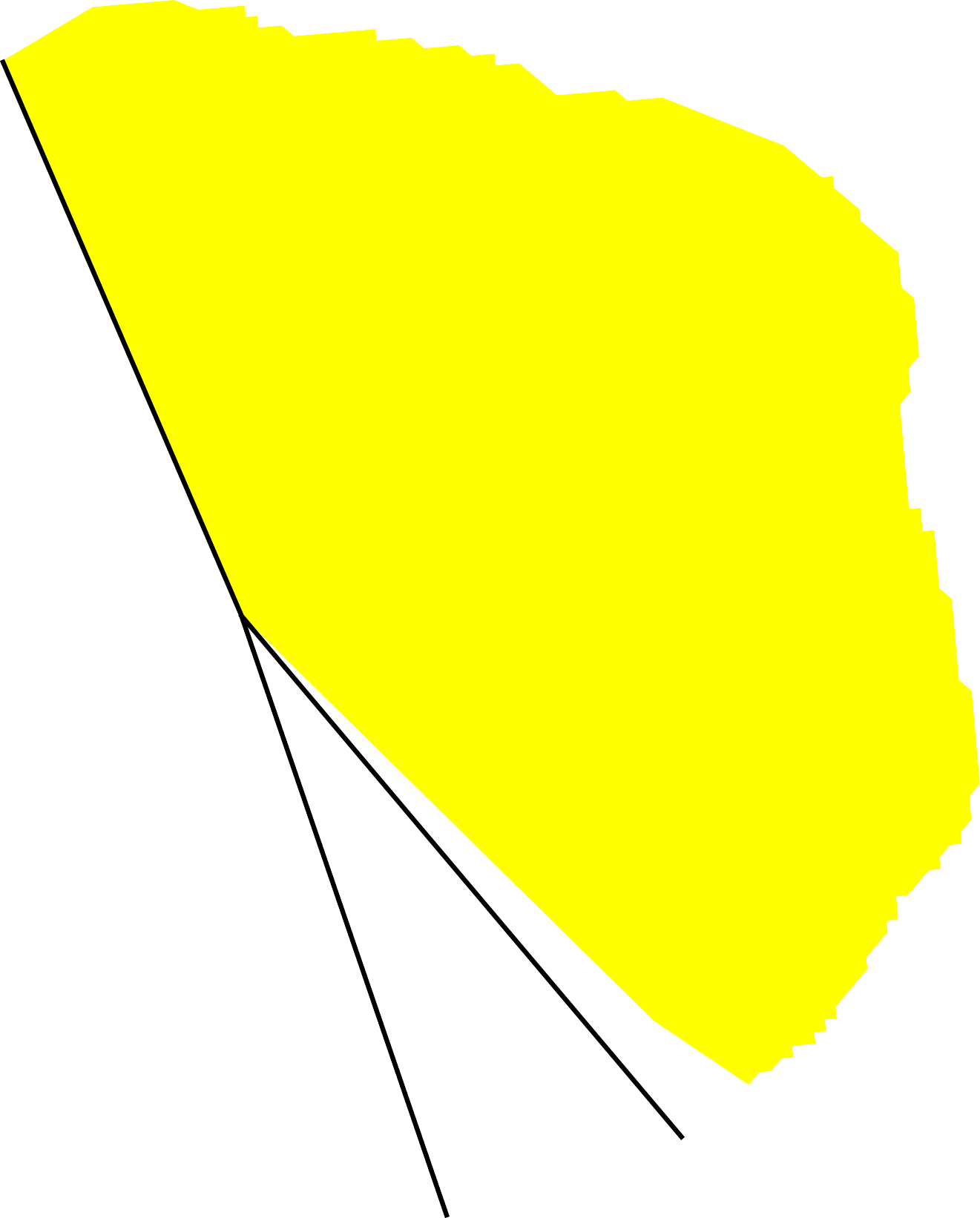}}%
    \put(0.0916563,0.76951334){\color[rgb]{0,0,0}\makebox(0,0)[lb]{\smash{$(\hat t,\hat x)$}}}%
    \put(0.32120106,0.1592447){\color[rgb]{0,0,0}\makebox(0,0)[lb]{\smash{$\sigma'$}}}%
    \put(0.55264945,0.16352784){\color[rgb]{0,0,0}\makebox(0,0)[lb]{\smash{$\sigma''$}}}%
  \end{picture}%
\spazioimmagine
   \def\svgwidth{\larghezzaimmagine} 
  \ifx\svgwidth\undefined
    \setlength{\unitlength}{387.67479301pt}
  \else
    \setlength{\unitlength}{\svgwidth}
  \fi
  \global\let\svgwidth\undefined
  \makeatother
  \begin{picture}(1,1.21123979)%
    \put(0,0){\includegraphics[width=\unitlength]{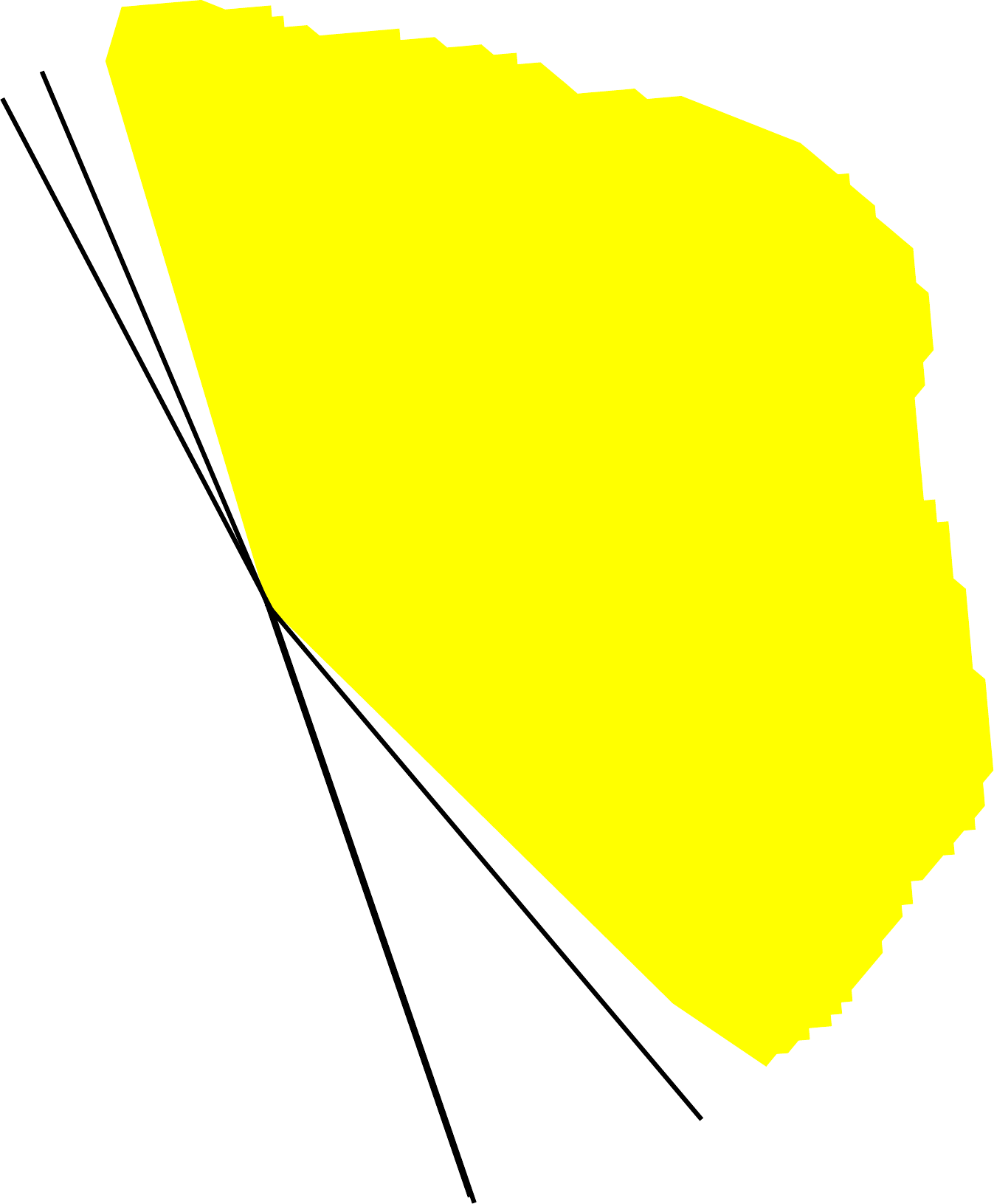}}%
    \put(0.00017299,0.98657857){\color[rgb]{0,0,0}\makebox(0,0)[lb]{\smash{$\sigma$}}}%
    \put(0.10904373,0.75207593){\color[rgb]{0,0,0}\makebox(0,0)[lb]{\smash{$(\hat t,\hat x)$}}}%
    \put(0.27481019,0.16681884){\color[rgb]{0,0,0}\makebox(0,0)[lb]{\smash{$\sigma'$}}}%
    \put(0.5620567,0.16681884){\color[rgb]{0,0,0}\makebox(0,0)[lb]{\smash{$\sigma''$}}}%
  \end{picture}%
\caption{Illustration of the flux function $\Phi_{\bar \imath,[a,b]}^{\nu,\out}(\{(\hat t,\hat x)\})$ of Lemma~\ref{L:waveBalances}}
\label{fig:interazioni}
\end{figure}
\begin{align*}
&-\sigma; \ \sigma\geq 0 &&
\text{\parbox{.6\linewidth}{If an $\bar \imath$-rarefaction wave exits.}}
\\
&\sigma'; \ \sigma'\leq 0& &
\text{\parbox{.6\linewidth}{If an $\bar \imath$-shock enters.}}
\\
&\sigma'+\sigma''; \ \sigma',\sigma''\leq 0&&
\text{\parbox{.6\linewidth}{If two $\bar \imath$-shocks out of $A^{\tau}_{[a,b]}$ interact.}}
\\
&\sigma';\ \sigma'\leq0& &
\text{\parbox{.6\linewidth}{If two $\bar \imath$-shocks, one in one out of $A^{\tau}_{[a,b]}$, interact.}}
\\
&-\sigma+\sigma';\ \sigma'\leq0,\ \sigma\geq 0& &
\text{\parbox{.6\linewidth}{If an $\bar \imath$-shock out of $A^{\tau}_{[a,b]}$ interacts with an $\bar \imath$-rarefaction wave, and a $\bar \imath$-rarefaction wave exits the region $A^{\tau}_{[a,b]}$.}}
\\
&\sigma';\ \sigma'\leq0& &
\text{\parbox{.6\linewidth}{If an $\bar \imath$-shock out of $A^{\tau}_{[a,b]}$ interacts with an $\bar \imath$-rarefaction wave, and no $\bar \imath$-rarefaction wave exits the region $A^{\tau}_{[a,b]}$.}}
\\
&-\sigma+\sigma';\ \sigma,\sigma'\geq0&&
\text{\parbox{.6\linewidth}{If an $\bar \imath$-rarefaction wave out of $A^{\tau}_{[a,b]}$ interacts with a $\bar \imath$-shock in $A^{\tau}_{[a,b]}$, and a $\bar \imath$-rarefaction wave exits the region $A^{\tau}_{[a,b]}$.}}
\\
&{\sigma'};\ \sigma'\geq0&&
\text{\parbox{.6\linewidth}{If an $\bar \imath$-rarefaction wave out of $A^{\tau}_{[a,b]}$ interacts with a $\bar \imath$-shock in $A^{\tau}_{[a,b]}$, and no $\bar \imath$-rarefaction wave exits the region $A^{\tau}_{[a,b]}$.}}
\\
&{\sigma' +\sigma''};\ \sigma'\geq0, \sigma'' \leq 0&&
\text{\parbox{.6\linewidth}{If an $\bar \imath$-rarefaction wave and $\bar \imath$-shock both out of $A^{\tau}_{[a,b]}$ interact, and no $\bar \imath$-rarefaction wave exits the region $A^{\tau}_{[a,b]}$.}}
\\
&-\sigma+\sigma'+\sigma'';\ \sigma,\sigma'\geq0,\sigma'' \leq 0&&
\text{\parbox{.6\linewidth}{If an $\bar \imath$-rarefaction wave and $\bar \imath$-shock both out of $A^{\tau}_{[a,b]}$ interact, and a $\bar \imath$-rarefaction wave exits the region $A^{\tau}_{[a,b]}$.}}
\end{align*}
The explanation is that when a physical interaction takes place at the boundary one has to balance the part of the waves that are out of the region we are looking at.
\emph{The assumption of genuine-nonlinearity of the $\bar \imath $-th characteristic field is what determines the sign of the big discontinuities, the $\bar\imath$-shocks}.
Considering that, the contribution of $\Phi^{\nu,\out}_{\bar \imath,[a,b]}$ may be positive only in the last four cases, when an $\bar \imath$-rarefaction wave comes from the outer region and hits a shock: in this case we now show that cancellation occurs and $\Phi^{\nu,\out}_{\bar \imath,[a,b]}$ is controlled by the amount of cancellation in $\mu_{\bar \imath}^{IC}$ at that point.
Indeed, suppose that an $\bar\imath$-rarefaction $\ell'$ hits an $\bar\imath$-shock $\ell''$: the value $u_{m}$ of $u^{\nu}$ between the two fronts should be the same, while the incidence condition implies $\dot\ell' > \dot\ell''$.
Moreover the speeds $\dot\ell' , \dot\ell''$ can be estimated by the intermediate value $u_{m}$ and the strength of the waves: one has
\[
\lambda_{\bar\imath}(u_{m})- \sigma_{\bar\imath}'/4
> 
\dot\ell'
>
\dot\ell''> \lambda_{\bar\imath}(u_{m})+3\sigma_{\bar\imath}''/4
\qquad
\Rightarrow
\qquad
-3\sigma_{\bar\imath}''\geq\sigma_{\bar\imath}'.
\]
This says that if a cancellation occurs, then the shock should be of the same size as the rarefaction.
Therefore the rarefaction even when not cancelled is controlled by the amount of cancellation.
Then
\[
\Phi^{\nu,\out}_{\bar \imath,[a,b]}(\{(\hat t,\hat x)\})\leq 3 \mu_{\bar \imath}^{IC}(\{(\hat t,\hat x)\}).
\] 
More formally, the `flux' $\Phi^{\nu,\out}_{\bar \imath,[a,b]}$ could be defined at once considering the exterior trace of $u^{\nu}$ on $A^{\tau}_{[a,b]}$.

\paragr{2) Single interval estimate.}
By the choice of $A^{\tau}_{[a,b]}$, which contains $I(\tau)$ but not $I(0)$, there is no flux concerning waves exiting at $\tau$ and no one concerning waves entering at $0$---even if an interaction takes place at $\tau$ or at $0$---because $v_{\bar \imath}(t)$ is continuous from the right. Indeed interactions at $t=0$ are not taken into account by $\mu^{\nu}_{\bar\imath}$ on ${{A^{\tau}_{[a,b]}}}$.
We recall that now we temporary neglect non-physical waves.

Piecing together the various physical interactions one obtains
\begin{equation}
\label{E:balclosedint}
\hat v^{\nu}_{\bar\imath}(\tau)-\hat v^{\nu}_{\bar\imath}(0)= \hat\mu^\nu_{\bar\imath}\big(A^{\tau}_{[a,b]}\big)+  \Phi^{\nu,\out}_{\bar \imath,[a,b]}(\Graph(a)\cup\Graph(b))
\qquad
 \Phi^{\nu,\out}_{\bar \imath,[a,b]}(B)\leq \OO\mu^{IC}\big(B\big),
\end{equation}
\end{subequations}
where $B\subset\Graph(a)\cup\Graph(b)$.
This formula holds also in the case $a=b$.
When removing the hat the same equation holds with an additional term uniformly controlled by $\varepsilon_{\nu}$, however we prefer to take non-physical wave into account after having obtained the balance for countably many intervals.

\paragr{3) Estimate on more intervals.}
Consider now two disjoint closed intervals $I_{0}=[x_{0}, x_{1}]$, $I_{1}=[x_{2}, x_{3}]$ and denote the relative selected $\bar\imath$-characteristics starting out from $x_{h}$ by $x_{h}(t)$, $I_{h}(t)=[x_{2h}(t), x_{2h+1}(t)]$.
Notice that the same equation as in~\eqref{E:balclosedint} holds immediately for $A^{t_{0},\tau}_{I_{1}\cup I_{2}}$ on time intervals where $x_{1}(t_{0}+\tau)<x_{2}(t_{0}+\tau)$.
We consider thus the case $x_{1}(t)=x_{2}(t)$ at time $\hat t\leq t_{0}+\tau$ and then, by the semigroup property, up to time $t_{0}+\tau$.

\begin{figure}[ht!] 
   \centering
   \def\svgwidth{.6\columnwidth} 
  \ifx\svgwidth\undefined
    \setlength{\unitlength}{776.052515pt}
  \else
    \setlength{\unitlength}{\svgwidth}
  \fi
  \global\let\svgwidth\undefined
  \makeatother
  \begin{picture}(1,0.50639937)%
    \put(0,0){\includegraphics[width=\unitlength]{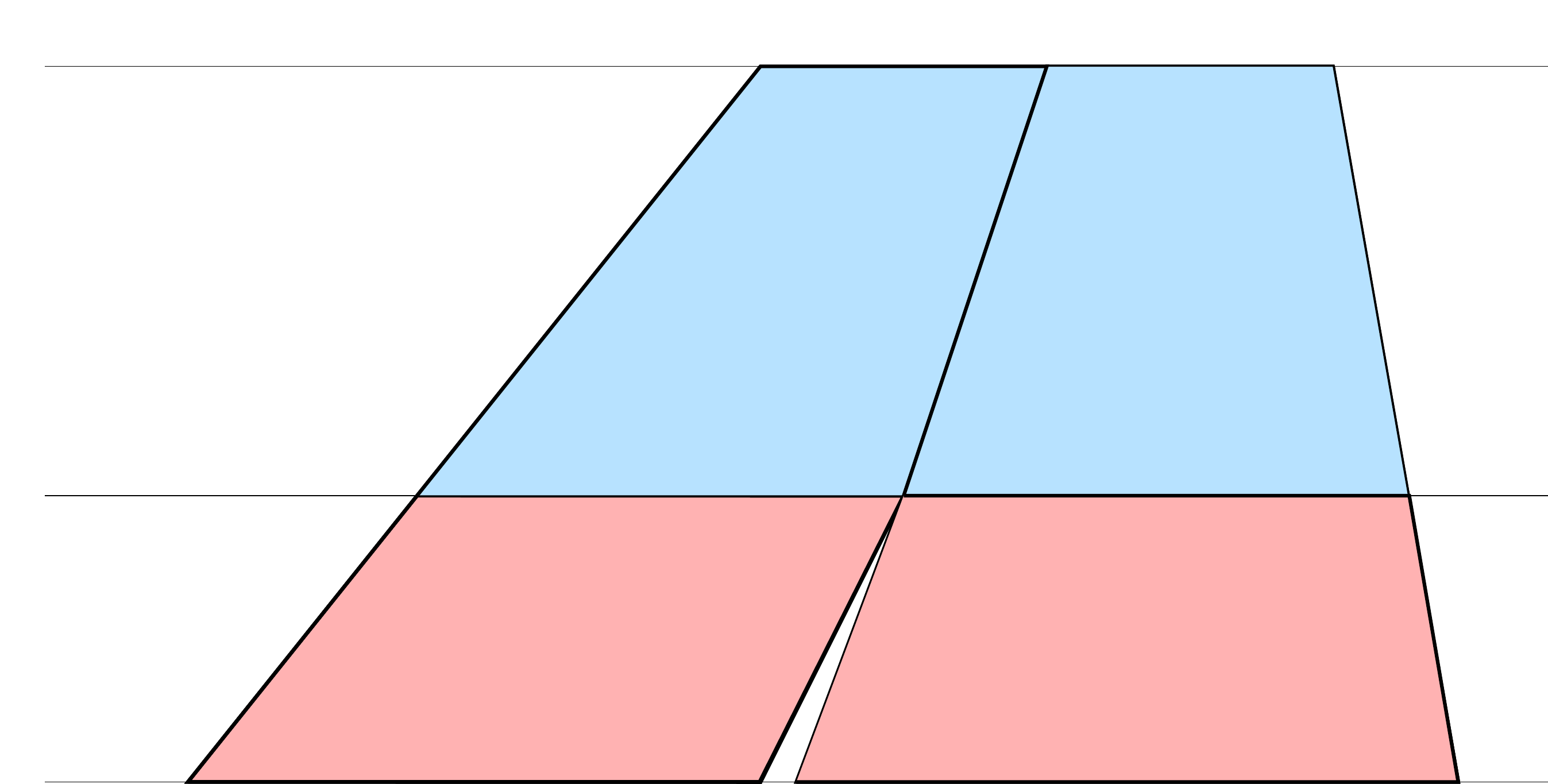}}%
    \put(0.94798139,0.03013226){\color[rgb]{0,0,0}\makebox(0,0)[lb]{\smash{$x_{3}$}}}%
    \put(0.58,0.48674007){\color[rgb]{0,0,0}\makebox(0,0)[lb]{\smash{$x_{1}^{t_{0}+\tau}=x_{2}^{t_{0}+\tau}$}}}%
    \put(0.0000867,0.03591211){\color[rgb]{0,0,0}\makebox(0,0)[lb]{\smash{$t_{0}$}}}%
    \put(0.0000867,0.22086717){\color[rgb]{0,0,0}\makebox(0,0)[lb]{\smash{$\hat t$}}}%
    \put(0.0000867,0.49829976){\color[rgb]{0,0,0}\makebox(0,0)[lb]{\smash{$t_{0}+\tau$}}}%
    \put(0.41623559,0.48674007){\color[rgb]{0,0,0}\makebox(0,0)[lb]{\smash{$x_{0}^{t_{0}+\tau}$}}}%
    \put(0.84972401,0.48674007){\color[rgb]{0,0,0}\makebox(0,0)[lb]{\smash{$x_{3}^{t_{0}+\tau}$}}}%
    \put(0.09834407,0.03013226){\color[rgb]{0,0,0}\makebox(0,0)[lb]{\smash{$x_{0}$}}}%
    \put(0.54339219,0.03013226){\color[rgb]{0,0,0}\makebox(0,0)[lb]{\smash{$x_{2}$}}}%
    \put(0.32005895,0.12399696){\color[rgb]{0,0,0}\makebox(0,0)[lb]{\smash{$A_{I_{0}}^{t_{0},\hat t - t_{0}}$}}}%
    \put(0.70183703,0.12399696){\color[rgb]{0,0,0}\makebox(0,0)[lb]{\smash{$A_{I_{1}}^{t_{0},\hat t - t_{0}}$}}}%
    \put(0.58485295,0.32159061){\color[rgb]{0,0,0}\makebox(0,0)[lb]{\smash{$A_{[x_{0}(\hat t),x_{3}(\hat t)]}^{\hat t, t_{0}+\tau-\hat t}$}}}%
    \put(0.45091466,0.03013226){\color[rgb]{0,0,0}\makebox(0,0)[lb]{\smash{$x_{1}$}}}%
  \end{picture}%
\caption[Illustration of Lemma~\ref{L:waveBalances}]{Illustration of Lemma~\ref{L:waveBalances}, balance on two intervals.}
\label{fig:bal2int}
\end{figure}
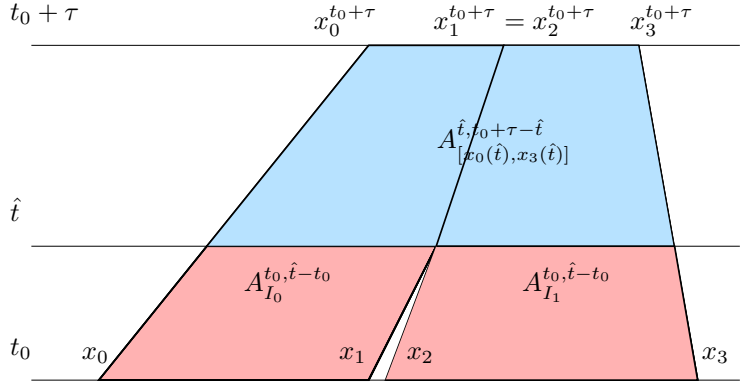
As illustrated in  Figure~\ref{fig:bal2int}, decompose $A^{t_{0},\tau}_{I_{0}}\cup A^{t_{0},\tau}_{I_{1}}$ into the regions $A^{t_{0}, \tau'}_{I_{0}},A^{t_{0}, \tau'}_{I_{1}}, A^{\hat t, \tau-\tau'}_{[x_{0}(\hat t), x_{3}(\hat t)]}$, with $\tau'=\hat t-t_{0}$, where one applies separately~\eqref{E:balclosedint}.
In order to shorten notations, fix again $t_{0}=0$:
\begin{align}
\notag
\hat v^{\nu}_{\bar\imath}(\tau) &\big([x_{0}(\tau),x_{3}(\tau)]\big)-\hat v^{\nu}_{\bar\imath}(0)\big(I_{0}\cup I_{1}\big)
\\
\notag
&=
\hat v^{\nu}_{\bar\imath}(\tau) \big([x_{0}(\tau),x_{3}(\tau)]\big)
-\hat v^{\nu}_{\bar\imath}(\hat t) \big([x_{0}(\hat t),x_{3}(\hat t)]\big)
\\
\notag
&\phantom{=}
+\hat v^{\nu}_{\bar\imath}(\hat t) \big([x_{0}(\hat t),x_{1}(\hat t)]\big)
+\hat v^{\nu}_{\bar\imath}(\hat t) \big([x_{2}(\hat t),x_{3}(\hat t)]\big)
-\hat v^{\nu}_{\bar\imath}(\hat t) \big(\{x_{1}(\hat t)=x_{2}(\hat t)\}\big)
\\
\notag
&\phantom{=}
-\hat v^{\nu}_{\bar\imath}(0)\big(I_{0})
- \hat v^{\nu}_{\bar\imath}(0)\big(I_{1}\big).
\end{align}
If one denotes briefly $\Phi^{\nu,\out}_{\bar \imath,I_{0}}((\hat t,\tau])=\Phi^{\nu,\out}_{\bar \imath,I_{0}}(A^{\hat t,\tau}_{I_{0}})$ and so on, the balances yield
\begin{align}
\notag
\hat v^{\nu}_{\bar\imath}(\tau) &\big([x_{0}(\tau),x_{3}(\tau)]\big)-\hat v^{\nu}_{\bar\imath}(0)\big(I_{0}\cup I_{1}\big)
\\
\notag
&=
\hat \mu^\nu_{\bar\imath}\big(A^{\hat t,\tau-\hat t}_{[x_{0}(\hat t),x_{3}(\hat t)]}\big)
+ \hat \mu^\nu_{\bar\imath}\big(A^{\hat t} _{I_{0}}\big)
+ \hat \mu^\nu_{\bar\imath}\big(A^{\hat t} _{I_{1}}\big)
-\hat v^{\nu}_{\bar\imath}(\hat t)\big(\{x_{1}(\hat t)\}\big)
 \\
\notag
& \qquad \qquad
 +\Phi^{\nu,\out}_{\bar \imath,I_{0}}((0,\hat t])
 +\Phi^{\nu,\out}_{\bar \imath,I_{1}}((0, \hat t])
+\Phi^{\nu,\out}_{\bar \imath,[x_{0}(\hat t),x_{3}(\hat t)]}((\hat t,\tau]).
\end{align}
It remains then to notice that as the regions $A^{\hat t} _{I_{0}}$, $A^{\hat t} _{I_{1}}$ overlap at the interaction point then
\begin{align}
\notag
&\hat \mu^\nu_{\bar\imath}\big(A^{\hat t,\tau-\hat t}_{[x_{0}(\hat t),x_{3}(\hat t)]}\big)
+ \hat \mu^\nu_{\bar\imath}\big(A^{\hat t} _{I_{0}}\big)
+ \hat \mu^\nu_{\bar\imath}\big(A^{\hat t} _{I_{1}}\big)
=\hat \mu^\nu_{\bar\imath}\big(A^{\hat t}_{I_{0}\cup I_{1}}\big) + \hat \mu^\nu_{\bar\imath} (\{(\hat t, x_{1}(\hat t))\}) .
\end{align}
Moreover, the waves entering at $(\hat t, x_{1}(\hat t))$ from outside the region $A^{\tau} _{I_{0}\cup I_{1}}=A^{\tau} _{I_{0}}\cup A^{\tau} _{ I_{1}}$ are counted both in $ \Phi^{\nu,\out}_{\bar \imath,I_{0}}((0,\hat t])$ and $\Phi^{\nu,\out}_{\bar \imath,I_{1}}((0, \hat t])$, but only once in $\Phi^{\nu,\out}_{\bar \imath,I_{0}\cup I_{1}}((0,\tau]) $, while the ones interacting at $(\hat t, x_{1}(\hat t))$ and coming from the interior of $A^{\tau} _{I_{0}\cup I_{1}}$ are counted precisely once in $\Phi^{\nu,\out}_{\bar \imath,I_{0}}((0,\hat t])+\Phi^{\nu,\out}_{\bar \imath,I_{1}}((0, \hat t])$: since the amount of these waves from the region in between is precisely $\hat v^{\nu}_{\bar\imath}(\hat t)\big(\{x_{1}(\hat t)\}\big) -  \hat \mu^\nu_{\bar\imath} (\{(\hat t, x_{1}(\hat t))\}) $, therefore
\begin{align}
\notag
 \Phi^{\nu,\out}_{\bar \imath,I_{0}}((0,\hat t])
 &+\Phi^{\nu,\out}_{\bar \imath,I_{1}}((0, \hat t])
+\Phi^{\nu,\out}_{\bar \imath,[x_{0}(\hat t),x_{3}(\hat t)]}((\hat t,\tau])
\\
\notag
& =\hat v^{\nu}_{\bar\imath}(\hat t)\big(\{x_{1}(\hat t)\}\big) -  \hat \mu^\nu_{\bar\imath} (\{(\hat t, x_{1}(\hat t))\}) 
+\Phi^{\nu,\out}_{\bar \imath,I_{0}\cup I_{1}}((0,\tau]) .
 \end{align}
Collecting the terms one obatins
\[
\hat v^{\nu}_{\bar\imath}(\tau) \big([x_{0}(\tau),x_{3}(\tau)]\big)-\hat v^{\nu}_{\bar\imath}(0)\big(I_{0}\cup I_{1}\big)
=
\hat \mu^\nu_{\bar\imath}\big(A^{\tau}_{I_{0}}\cup A^{\tau} _{I_{1}}\big)
 +\Phi^{\nu,\out}_{\bar \imath,I_{0}\cup I_{1}}((0,\tau])
.
\]

The reasoning applied to two intervals $I_{0}$, $I_{1}$ can be applied as well to any finite number, obtaining the analogous equation.
Since the flux can be estimated by the interaction-cancellation measure, one finds then that for all $M\in\N$ and real intervals $I_{1},\dots , I_{M}$
\[
\hat v^{\nu}_{\bar\imath}(\tau) \big(I_{1}(\tau)\cup\dots \cup I_{M}(\tau)\big)-\hat v^{\nu}_{\bar\imath}(0)\big(I_{1}\cup\dots \cup I_{M}\big)
\leq
(\hat\mu^\nu_{\bar\imath}+ \OO\mu^{IC}_{\nu})\big(A^{\tau}_{I_{1}}\cup \dots \cup A^{\tau} _{I_{M}}\big) .
\]

Since the total strength of non-physical fronts at each finite time is controlled by $\varepsilon_{\nu}$, as well as the mass $|\rho^{\nu}|$ of $\mu_{\bar \imath} ^{\nu}$ due to interactions involving non-physical waves and $\mu^\nu_{\bar\imath}$ (Lemma~\ref{L:waveBalRadon}), from above
\begin{align}
v^{\nu}_{\bar\imath}(\tau) \big(I_{1}(\tau)\cup\dots \cup I_{M}(\tau)\big)-v^{\nu}_{\bar\imath}(0)\big(I_{1}\cup\dots \cup I_{M}\big)
&\leq
 (\mu^\nu_{\bar\imath}+\OO\mu^{IC})\big(A^{\tau}_{I_{1}\cup\dots \cup I_{M}}\big)
 +\OO\varepsilon_{\nu}
 \\
 \notag
 &\leq \OO\Big( \mu^{IC}\big(A^{\tau}_{I_{1}\cup\dots \cup I_{M}}\big)
 +\varepsilon_{\nu}\Big) .
\end{align}

\paragr{4) Balance for the jump part.}
For the `jump part' by~\eqref{EG:mujump} one can repeat the same argument as above with the relative measures, but the fluxes are non-positive because only $i$-shocks are involved, so that one does not need to estimate the positive part (which vanishes) of the fluxes by the interaction-cancellation measure: therefore
\begin{equation}
\label{E:jumpbalance}
v^{\nu}_{\bar\imath,\jump}(\tau) \big([a(\tau),b(\tau)]\big)-v^{\nu}_{\bar\imath,\jump}(0)\big([a,b]\big)
=
\mu^\nu_{\bar\imath,\jump}\big(A^{\tau}_{[a,b]}\big)
 +\Phi^{\nu,\out,\jump}_{\bar \imath,[a,b]}(A^{\tau}_{[a,b]})
\end{equation}
holds with new fluxes $\Phi^{\nu,\out,\jump}_{\bar \imath,[a,b]}$ which take into account only the contribute in $\Phi^{\nu,\out}_{\bar \imath,[a,b]}$ due to $i$-shocks in $\mathcal J^{\nu,i}_{(\varepsilon_{0},\varepsilon_{1})} $. In particular this flux is non-positive. Explicitly the new flux at the boundary $\Phi^{\nu,\out,\jump}_{\bar \imath,[a,b]}$ is
\begin{equation*}
\begin{cases}
\sigma'; \ \sigma'\leq 0& 
\text{\parbox{.75\linewidth}{If an $\bar \imath$-shock of strength $\sigma'$ in $\mathcal J^{\nu,i}_{(\varepsilon_{0},\varepsilon_{1})} $ enters, possibly interacting with a front wave either in $A^{\tau}_{[a,b]}$ or not belonging to $\mathcal J^{\nu,i}_{(\varepsilon_{0},\varepsilon_{1})} $.}}
\\
\sigma'+\sigma''; \ \sigma',\sigma''\leq 0&
\text{\parbox{.75\linewidth}{If two $\bar \imath$-shocks out of $A^{\tau}_{[a,b]}$ and in $\mathcal J^{\nu,i}_{(\varepsilon_{0},\varepsilon_{1})} $ interact.}}
\\
 0&
\text{\parbox{.75\linewidth}{Otherwise.}}
\end{cases}
\end{equation*}
The generalization to countably many intervals holds precisely as before.

\paragr{5) Balance for the continuous part.}
By subtracting~\eqref{E:jumpbalance} to~\eqref{E:balclosedint} one finds the analogous estimate for $v_{\cont}$, which again holds in the same way for countably many intervals, with a new flux at the boundary $\Phi^{\nu,\out,\cont}_{\bar \imath,[a,b]}(A^{\tau}_{[a,b]})$ which is the difference $\Phi^{\nu,\out}_{\bar \imath,[a,b]}-\Phi^{\nu,\out,\jump}_{\bar \imath,[a,b]}$: notice indeed that one can take advantage of \[\Phi^{\nu,\out}_{\bar \imath,[a,b]}-\Phi^{\nu,\out,\jump}_{\bar \imath,[a,b]}\leq [\Phi^{\nu,\out}_{\bar \imath,[a,b]}-\Phi^{\nu,\out,\jump}_{\bar \imath,[a,b]}]^{+}\leq [\Phi^{\nu,\out}_{\bar \imath,[a,b]}]^{+} \leq \mu_{\nu}^{IC} .
\]
Indeed, the negative terms in $\Phi^{\nu,\out,\jump}_{\bar \imath,[a,b]}$ cancel positive ones in $\Phi^{\nu,\out}_{\bar \imath,[a,b]}$, without adding a positive contribution in the difference.
The last equation in the statement holds then by definition of $\mu_{\bar\imath,\nu}^{ICJ}=\mu_{\nu}^{IC}+| \mu^\nu_{\bar\imath,\jump} |$, which controls $ \mu^\nu_{\bar\imath}- \mu^\nu_{\bar\imath,\jump}$.
\end{proof}

\subsection{The decay estimate}
We prove in this section an estimate analogous to the decay of positive waves (Pages 210-216 in~\cite{Bressan}), but for the \emph{negative} part of  $(v^{}_{\bar\imath})_{\cont}$, under the assumption that the $\bar\imath$-characteristic field is genuinely non-linear.
We bound from below the continuous part of the waves of a semigroup solution $u$ on a Borel set $B$ at time $t$ by the wave$\backslash$jump wave balance measures $\mu_{\bar \imath}\backslash\mu_{\bar \imath,\jump}$ on a strip around $t$ of arbitrary height $\tau<t$ and the Lebesgue measure of $B$ divided by $\tau$: together with the previous result due to Bressan we obtain
\[
|(v_{\bar\imath})_{\cont}|(t)(B) \leq 
\OO\Ll^{1}(B)/\tau+
\OO\big( \mu_{\bar\imath}^{ICJ}\big)\big([t-\tau,t+\tau]\times\R\big) 
\qquad
\forall t>\tau>0.
\]
We remind that we defined $\mu_{\bar\imath}^{ICJ}$ as the $w^{*}$-limit of the measures $\mu^{ICJ}_{\bar\imath,\nu}$ in~\eqref{E:muICJinu}.

We remark that here there is no sharpness purpose.
We first prove an estimate from below for the `absolutely continuous' part of $v^{\nu}_{\bar\imath}$.
As a consequence of Corollary~\ref{C:weakConvergences}, and Remark~\ref{R:convrestr}, when passing to the limit we get the claim (Lemma~\ref{L:decaEstOnIntervals}).

\begin{lemma}[Approximate decay estimate on intervals]
\label{L:approxestimate}
Assume the $\bar \imath$-th characteristic field is genuinely non-linear.
Then for any disjoint closed intervals $\{I_{h}\}_{h\in\N}$ and $\tau, t_{0}> 0$ one has the bound
\[
-v^{\nu}_{\bar \imath,\cont}(t_{0})\big(I_{1}\cup\dots\cup I_{M}\big) 
\leq 
\OO\Big\{\Ll^{1}(I_{1}\cup\dots\cup I_{M})/\tau+ \mu_{\bar\imath,\nu}^{ICJ}( A^{t_{0},\tau}_{I_{1}\cup\dots\cup I_{M}}) +\varepsilon_{\nu}+\varepsilon_{1} \Big\},
\]
where $A^{t_{0},\tau}_{I_{1}\cup\dots\cup I_{M}}$ is the region bounded by generalized $\bar\imath$-characteristics in~\eqref{E:At0tau}.
\end{lemma}
\begin{remark}
\label{R:approxestimate}
Under the assumptions of the above lemma, an analogous statement holds for any Borel set $B$ if one considers its evolution by minimal generalized $\bar \imath $-th characteristics (or a different selection with the semigroup property). Denoting the generalized $\bar \imath $-th characteristics $\{y^{\nu}(t;t_{0},x)\}_{x\in B}$ and defining $A_{B}={\{y^{\nu}(t;t_{0},x)\}_{t_{0}<t\leq t_{0}+\tau}^{x\in B}}$ one has
\[
-v^{\nu}_{\bar \imath,\cont}(t_{0})\big(B\big) 
\leq 
\OO\Big\{\Ll^{1}(B)/\tau+ \mu_{\bar\imath,\nu}^{ICJ}( \overline{A_{B}}) +\varepsilon_{\nu}+\varepsilon_{1} \Big\}.
\]
Indeed, this holds true for countably many intervals by Lemma~\ref{L:approxestimate} and the observation that
\[
\mu_{\bar\imath,\nu}^{ICJ}( { A^{t_{0},\tau}_{[a,b]}}) = \mu_{\bar\imath,\nu}^{ICJ}( {A_{[a,b]}}) = \mu_{\bar\imath,\nu}^{ICJ}(\overline {A_{[a,b]}}\cap\{t>t_{0}\}),
\]
for any choice of characteristics starting out from $[a,b]$ and having the semigroup property.
Since on the r.h.s.~we have a nonnegative Radon measure, by inner/outer regularity it can be extended to Borel sets.
\end{remark}
\begin{proof}[Proof of Lemma~\ref{L:approxestimate}]
We adapt the argument in~\cite{Bressan} (\S10.2) about the decay of positive waves, whose main steps we briefly recall here.
Since the generalization to countably many intervals is analogous to the one of Lemma~\ref{L:waveBalances}, we consider the case of a single interval $I=[a,b]$.

The heuristic is the following.
Suppose $v^{\nu}_{\bar \imath,\cont}(t_{0})(I)<0$ and there are only $\bar \imath$-fronts in $I$, which is the main case we would like to manage.
Then initially $\Ll^{1}(I(t))$ decreases at a rate at least $v^{\nu}_{\bar \imath}(t_{0})(I)/4$; if $\Ll^{1}(I(t))$ keeps on decreasing at least of that rate between times $t_{0}$, $t_{0}+\tau$, one can estimate the initial value $\Ll^{1}(I)$ from below as stated just by integrating this differential relation.
Otherwise, interactions must take place in order to decrease the rate, with a special care for the ones at the boundary.

For simplicity of notation, we set $t_{0}=0$ and we consider the $\bar\imath$-th wave measures omitting the indices $\bar \imath$.
Let the interval $[a,b]$ evolve by the minimal forward characteristics: its length at time $t$ is
\[
z(t)=b(t)-a(t) .
\]
This function is absolutely continuous and it satisfies $\dot z(t)=\bar\lambda(t,b(t))-\bar\lambda(t,a(t))$ for a.e.~$t$. 

\emph{Estimate on the influence of the other families on the $\bar\imath$-wave measure.} At Page~213 of~\cite{Bressan} it is explicitly defined a function $\Phi$ which contains the variation of speed due to the presence of waves of other characteristic families.
$\Phi$ is piecewise Lipschitz continuous, non-decreasing, with discontinuities only at interaction times.
More precisely one has the estimate
\begin{gather}
\label{E:stimazBressan}
|\dot z(t)+\xi(t) - v^{\nu}(t)|\leq \OO(\varepsilon_{\nu} + \dot \Phi(t) z(t) )
\end{gather}
where $v^{\nu}(t)$ counts the variation of speed due to the waves of the same family, and $\xi$ corrects it taking into account waves at the boundary: we adopt the notation
\begin{gather*}
v^{\nu}_{(\cont)}(t):= v^{\nu}_{(\cont)}(t)([a(t),b(t)]),
\\
\xi(t):=(\bar\lambda(t,a(t))-\bar\lambda(t,a(t)^{-}))+(\bar\lambda(t,b(t)^{+})-\bar\lambda(t,b(t)))
.
\end{gather*}

\emph{Case 1.} If $\dot z(t) - \dot\Phi(t)z(t) < v^{\nu}_{\cont}(0)/4$ for all $t$, then
\[
d/dt\big(e^{-\int_{0}^{t}\dot\Phi}z(t) \big)
=
e^{-\int_{0}^{t}\dot\Phi}\big\{\dot z(t) - \dot\Phi(t)z(t)\} 
< 
e^{-\int_{0}^{t}\dot\Phi}v^{\nu} \cont(0)/4
\leq 
v^{\nu} \cont(0)/4
\]
and integrating the inequality between times $0$ and $\tau$, since $z(t)\geq 0$, one has
\[
-z(0)
\leq 
e^{-\int_{0}^{\tau}\dot\Phi}z(\tau)-z(0)
\leq 
\tau v_{\cont}^{\nu}(0)/4 
.
\]

\emph{Case 2.} Suppose instead $\dot z(\bar t) - \dot\Phi(\bar t)z(\bar t) \geq v^{\nu}_{\cont}(0)/4$ at some time $\bar t$.
\\
By~\eqref{E:stimazBressan} we have that
\[
\dot z(t) - \dot \Phi(t) z(t)\leq   v^{\nu}(t) -\xi(t) + \OO\varepsilon_{\nu} .
\]
Observing that $\bar\lambda(t,{a(t)}^{})$ is a suitable mean of $\lambda(t,{a(t)}^{\pm}) $ and $ v^{\nu}_{\jump}((t,a(t)))= \lambda(t,a(t)^{+}) -\lambda(t,{a(t)}^{-}) $ by~\eqref{E:sigmalushock}, one derives the bound 
\begin{align*}
\frac{4}{3}\xi(t)
&\stackrel{\phantom{\eqref{E:jumpbalance}}}{\geq }
 v^{\nu}_{\jump}((t,a(t)))+  v^{\nu}_{\jump}((t,b(t))) - 2\varepsilon_{1} 
\\
&\stackrel{\phantom{\eqref{E:jumpbalance}}}{\geq} 
v^{\nu}_{\jump}(t) - 2\varepsilon_{1}
.
\end{align*}
The two inequality, being $v^{\nu}_{\jump}$ negative, yield immediately that
\[
\dot z(t) - \dot \Phi(t) z(t)\leq
 v^{\nu}_{\cont}(t)+v^{\nu}_{\jump}(t) -\xi(t) + \OO\varepsilon_{\nu}
 \leq 
 v^{\nu}_{\cont}(t) + \OO\varepsilon_{\nu}+2\varepsilon_{1}.
\]
At time $\bar t$, taking also into account that we are assuming $\dot z(\bar t) - \dot\Phi(\bar t)z(\bar t) \geq v^{\nu}_{\cont}(0)/4$, we get thus
\[
v^{\nu}_{\cont}(0)/4\leq v^{\nu}_{\cont}(\bar t) + \OO\varepsilon_{\nu}+2\varepsilon_{1}
\] 
By the balance for the continuous part in the proof of Lemma~\ref{L:waveBalances} we obtain
\[
v^{\nu}_{\cont}(0)/4\leq v^{\nu}_{\cont}(0)  + \OO\mu^{ICJ}_{\nu}(A^{\tau}_{[a,b]}) + \OO\varepsilon_{\nu}+2\varepsilon_{1} .
\]
%

\emph{Conclusion.}
Collecting the two cases one finds the bound uniform in $\nu$
\[
v^{\nu}_{\cont}(0)
\geq 
-\OO \bigg\{ (b-a)/\tau
 +\mu^{ICJ}_{\nu}(A^{\tau}_{[a,b]}) 
 +\varepsilon_{1}  +\varepsilon_{\nu} \bigg\}.
\]
This shows the claim for a single closed interval.
Similarly to what is done in Lemma~\ref{L:waveBalances}, the argument extends by direct computation to the union of countably many closed intervals.
\end{proof}

We finally arrive to the decay estimate for the negative part of $(v_{\bar\imath})_{\cont}$.
The characteristics of the semigroup solution $u$ we get in the limit in general are no more minimal, but just generalized $\bar\imath$-characteristics: this is why in~\eqref{E:decayEst} we have chosen to state it enlarging the set $A_{B}$ up to an horizontal strip.

\begin{lemma}[Decay estimate on intervals]
\label{L:decaEstOnIntervals}
Assume the $\bar \imath$-th characteristic field is genuinely non-linear.
Then there is a choice of generalized $\bar \imath $-th characteristics $\{y^{}(t;s,x)\}_{x\in \R}^{s\geq0}$ of $u$ such that for any Borel set $B\subset\R$ and $t_{0}>\tau>0$ the following estimate of the $\bar\imath$-th wave holds
\[
-(v^{}_{\bar \imath})_{\cont}(t_{0})\big(B\big) 
\leq 
\OO\Big\{\Ll^{1}(B)/\tau+ \mu_{\bar\imath}^{ICJ}( \overline{A_{B}}) \Big\}
\qquad
A_{B}={\{y^{}(t;t_{0},x)\}_{t_{0}\leq t\leq t_{0}+\tau}^{x\in B}}.
\]
\end{lemma}

\begin{remark}
By the Jordan decomposition of $(v^{}_{\bar \imath})_{\cont}(t_{0})$ and the arbitrariness of $B$, this is precisely a lower bound for the negative part of $(v^{}_{\bar \imath})_{\cont}(t_{0})$.
\end{remark}

\begin{proof}
As usual, simplify notations by setting $t_{0}=0$ and omitting $\bar\imath$.

When $t_{0}$ was not a time in $\Theta$ of Theorem~\ref{T:pwconvergence}, by Remark~\ref{R:convrestr} the `continuous part' of $v^{\nu}(0)$ $w^{*}$-converge to the continuous part of $v^{}$.
Nevertheless, we do not know whether the positive and negative part $[v^{\nu}_{}(0)]^{\pm}$ of $v^{\nu}(0)$, and $[v^{\nu}_{\bar \imath,\cont}(0)]^{\pm}$, converge to the positive and negative part in the Jordan decomposition of $v(0)$, and $(v(0))_{\cont}$.
However, if one defines measures $[\bar v]^{\pm}$ by the relation
\[
[v^{\nu}_{\cont}(0)]^{\pm} \xrightarrow{w^{*}} [\bar v]^{\pm},
\]
where the convergence is obtained by compactness and up to a subsequence, then
\[
[\bar v]^{+}-[\bar v]^{-}= (v^{}_{}(0))_{\cont}
\qquad
[\bar v]^{+}\geq [(v^{})_{\cont}(0)]^{+}
\qquad
[\bar v]^{-}\geq [(v^{})_{\cont}(0)]^{-}.
\]
In general $[\bar v]^{\pm}$ are not orthogonal.

Since the atoms of $[\bar v]^{\pm}$ are at most countably many, one can consider first finitely many intervals $I_{1},\dots,I_{M}$ whose boundary has $0$ $|\bar v|$-measure, where $|\bar v|=[\bar v]^{+}+[\bar v]^{-}$.
Then the measures of $J=I_{1}\cup\dots\cup I_{M}$ converge:
\[
[v^{\nu}_{\cont}]^{\pm}(0)\big(J \big) \xrightarrow{} [\bar v]^{\pm} \big(  J \big),
\qquad
(v^{\nu})_{\cont}(0) \big( J\big) \xrightarrow{}  v^{}_{\cont}(0) \big(I \big)
 .
\]

By Lemma~\ref{L:approxestimate} and the following remark, the inequality
\[
-v^{\nu}_{\cont}(0) \big( J \big) 
\leq 
\OO\Big\{\Ll^{1}(J)/\tau+ \mu_{\nu}^{ICJ}( \overline{A^{\nu}_{J}}) +\varepsilon_{\nu}+\varepsilon_{1} \Big\}.
\]
holds with the minimal generalized $\bar\imath$-characteristics $\{y^{\nu}(t;t_{0},x)\}_{x\in J}$ of $u^{\nu}$ starting from $J$.

In the $\nu$-limit, by Ascoli-Arzela\`a theorem the selected $\nu$-characteristics $\{y^{\nu}(t;t_{0},x)\}_{x\in J}$ converge to some generalized $\bar\imath$-characteristics $\{y^{}(t;t_{0},x)\}_{x\in J}$ of $u$.
By the upper continuity of Borel probability measures on compact sets one can derive
\begin{align*}
-v_{\cont}(0)\big(J\big) 
&=  
-\lim_{\nu}  v^{\nu}_{\cont}(0) \big(J \big) 
\\
&\leq  
\liminf_{\nu} \OO\Big\{\Ll^{1}(J)/\tau+ \mu_{\nu}^{ICJ}( \overline{A^{\nu}_{J}}) +\varepsilon_{\nu}+\varepsilon_{1} \Big\}
\\
&\leq  
\OO\Big\{\Ll^{1}(J)/\tau+ \mu_{}^{ICJ}( \overline{A^{}_{J}})  \Big\}.
\end{align*}
The same results hold for open sets by approximation from the interior.
Then outer regularity yields the inequality for any Borel set.

When $t_{0}$ is one of the countably many points of $\Theta$ in Theorem~\ref{T:pwconvergence}, one should take into account that $t\mapsto(v_{\bar \imath})_{\jump}(t)$ is continuous from the right, and thus $(v_{\bar \imath})_{\cont}$.
The thesis can thus be obtained by a diagonal argument for a sequence $t_{k}\downarrow t_{0}$.
\end{proof}

\begin{corollary}
\label{C:final_estimate}
Consider a strictly hyperbolic system of conservation laws as~\eqref{E:consLaw} and let the $\bar\imath$-th characteristic field be genuinely non-linear,~\eqref{E:genNonl}.
By Lemma~\ref{L:decaEstOnIntervals} and the decay of positive waves in Theorem~10.3 of~\cite{Bressan}, one has that there exists a positive constant $C$ such that for every $t>\tau>0$, every Borel subset $B$ of $\R$ and every solution $u$ obtained as a limit of the front tracking approximation, the continuous part of the measures $v_{\bar\imath}$ satisfies
\[
|(v_{\bar\imath})_{\cont}|(t)(B) \leq 
C\Ll^{1}(B)/\tau+
C\big( \mu_{\bar\imath}^{ICJ}\big)\big([t-\tau,t+\tau]\times\R\big) .
\]
\end{corollary}

\appendix

\section{Table of notations}

\newcommand{\secondcolumn}{.8\textwidth}

\begin{tabbing}
\hspace*{0.5cm}\=$\lambda_1(u)<\dots <\lambda_\dmns(u)$ \=\kill
   \> $\Ll^{1}$, $\Ll^{2}$
 \> \parbox{\secondcolumn}{The Lebesgue measure of dimension respectively $1$ and $2$}
 \\
 \>$\OO$
 \> A positive constant which can be uniformly bounded in the limiting index
\\
    \> $[\eta]^{+}$, $[\eta]^{-}$, $|\eta|$
 \> \parbox{\secondcolumn}{Respectively positive part, negative part and variation of a measure $\eta$}
 \\
\> $(\eta)_{\jump}$, $(\eta)_{\cont}$
 \> The jump part and the continuous part of a measure $\eta $
\\
\> $\Omega$
 \> \parbox{\secondcolumn}{Open subset of $\R^{\dmns}$}
\\
\> $f$
 \> \parbox{\secondcolumn}{Flux function for a strictly hyperbolic system of conservation laws}
 \\
\> $A$ 
 \> The Jacobian matrix $A=\Dif f$
 \\
\> $\lambda_i(u)$
 \> The $i$-th eigenvalue of $A=\Dif f$, by strict hiperbolicity $\lambda_1<\dots <\lambda_\dmns$
 \\
\> $\lambda_i^{(\nu)}(t,x)$, $\bar\lambda_{i}(t,x)$
 \> $\lambda_i$ composed with $u^{(\nu)}$ and a suitable mean at jump points (Page~\pageref{eigenvectors} and~\eqref{E:lambdalrnu})
 \\
 \> $l$, $\tilde l^{(\nu)}$, $r$, $\tilde r^{(\nu)}$
 \> Left and right eigenvectors of $A=\Dif f$, suitably normalized (Page~\pageref{eigenvectors})
\\
\> $\sigma_1,\dots,\sigma_\dmns, \Lambda, \Psi_{i}$
 \> Strengths of the waves and maps in the Riemann problem, see Page~\pageref{T:sigmaLambda} and~\eqref{E:sigmalu}
\\
\> $u$
 \> \parbox{\secondcolumn}{Mainly the variable in $\R^{\dmns}$ or the semigroup solution to~\eqref{E:consLaw}}
 \\
\> $u^{\nu}$
 \> \parbox{\secondcolumn}{The $\nu$-front-tracking approximation, see recalls in Section~\ref{S:frontTrackingApp}}
 \\
\> $u_{x}$, $u_{t}$
 \> \parbox{\secondcolumn}{The (measure) derivatives of the $\BV_{\loc}(\R^{+}\times\R)$ function $u$}
 \\
\> $u(t)$
\> \parbox{\secondcolumn}{The restriction of the function $u$ at time $t\geq 0$, see Page~\pageref{lab:normalizationu}}
 \\
\> $v_{1}^{(\nu)},\dots, v_{\dmns}^{(\nu)}$
 \> Measures of the wave decomposition~\eqref{E:waveDec} of $u_{x}^{(\nu)}$ along $\tilde r_{1}^{(\nu)},\dots ,\tilde r_{\dmns}^{(\nu)}$
 \\
\> $ v_{i}^{(\nu)}(t)$
 \> Conditional measures of $v_{i}^{(\nu)}$ in the disintegration w.r.t.~time, see e.g.~Page~\pageref{T:continuousEstimate}
 \\
\> $v^{\nu}_{i,\jump},v^{\nu}_{i,\cont}$
\> Parts of $v^{\nu}_{i}$ converging to the jump$\backslash$continuous part of $v_{i}$ (Section~\ref{S:frontTrackingApp}, Page~\pageref{Ss:balchar})
 \\
 \>
 $\mathcal J^{\nu,i}_{(\varepsilon_{0},\varepsilon_{1})}, \mathcal J$
\> The `jump sets' of $u$ and $u^{\nu}$ of Section~\ref{S:frontTrackingApp}
\\
 \> $\mu_{\nu}^{I}$, $\mu_{\nu}^{IC}$
 \> \parbox{\secondcolumn}{Interaction and interaction-cancellation measures, see~\eqref{E:intrInterCanc}}
\\
 \> $\mu_{}^{I}$, $\mu_{}^{IC}$
 \> \parbox{\secondcolumn}{A $w^{*}$-limit of $\mu_{\nu}^{I}$, $\mu_{\nu}^{IC}$}
\\
  \> $\mu_{i}^{(\nu)},\mu_{i,\jump}^{(\nu,(\varepsilon_{0},\varepsilon_{1}))}$
 \> \parbox{\secondcolumn}{Balance and wave balance measures defined in Section~\ref{Ss:wavebalmeas}}
 \\
  \> $\mu^{ICJ}_{(\bar\imath,\nu)}$
 \> \parbox{\secondcolumn}{The interaction-cancellation-jump balance measures in~\eqref{E:muICJinu}, and a relative $w^{*}$-limit}
\\
\> $\Phi_{\bar \imath,[a,b]}^{\nu,\out, (\jump\backslash\cont)}$
 \> The fluxes introduced in the proof of Lemma~\ref{L:waveBalances}
 \\
   \>$A^{t_{0},\tau}_{[a,b]}$
 \> The region bounded by selected (e.g.~minimal) $i$-characteristics in~\eqref{E:At0tau}
\end{tabbing}


\end{document}

%% file: balancearxiv.pdf_tex

\begingroup
  \makeatletter
  \providecommand\color[2][]{%
    \errmessage{(Inkscape) Color is used for the text in Inkscape, but the package 'color.sty' is not loaded}
    \renewcommand\color[2][]{}%
  }
  \providecommand\transparent[1]{%
    \errmessage{(Inkscape) Transparency is used (non-zero) for the text in Inkscape, but the package 'transparent.sty' is not loaded}
    \renewcommand\transparent[1]{}%
  }
  \providecommand\rotatebox[2]{#2}
  \ifx\svgwidth\undefined
    \setlength{\unitlength}{268.625pt}
  \else
    \setlength{\unitlength}{\svgwidth}
  \fi
  \global\let\svgwidth\undefined
  \makeatother
  \begin{picture}(1,0.75644486)%
    \put(0,0){\includegraphics[width=\unitlength]{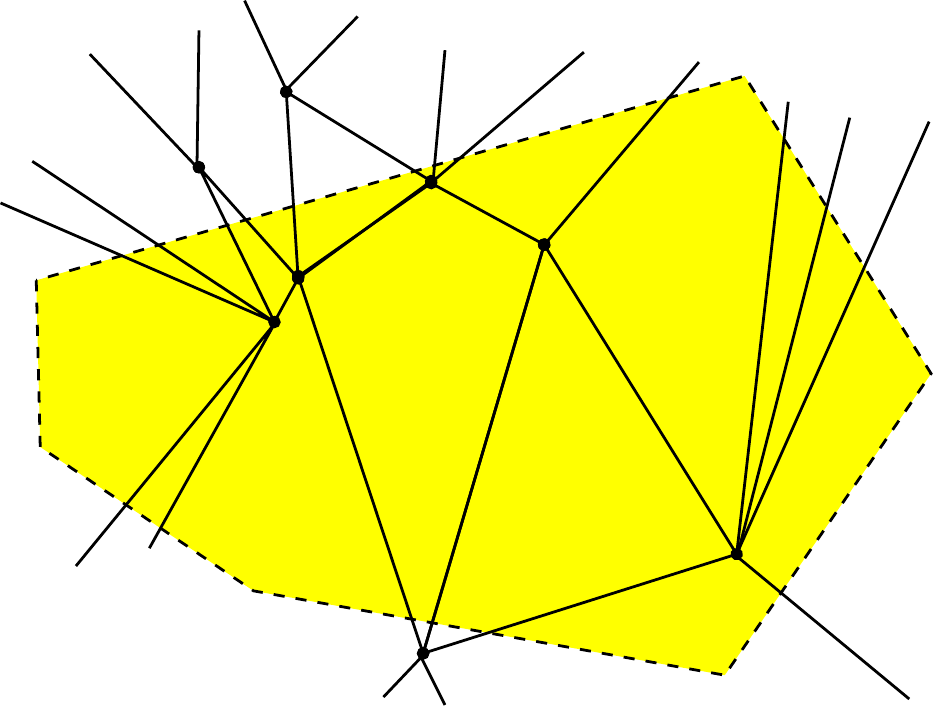}}%
    \put(0.27074384,0.18359202){\color[rgb]{0,0,0}\makebox(0,0)[lb]{\smash{$\Gamma$}}}%
  \end{picture}%
\endgroup

%% file: lemmaarxiv.pdf_tex

\begingroup
  \makeatletter
  \providecommand\color[2][]{%
    \errmessage{(Inkscape) Color is used for the text in Inkscape, but the package 'color.sty' is not loaded}
    \renewcommand\color[2][]{}%
  }
  \providecommand\transparent[1]{%
    \errmessage{(Inkscape) Transparency is used (non-zero) for the text in Inkscape, but the package 'transparent.sty' is not loaded}
    \renewcommand\transparent[1]{}%
  }
  \providecommand\rotatebox[2]{#2}
  \ifx\svgwidth\undefined
    \setlength{\unitlength}{267.585712pt}
  \else
    \setlength{\unitlength}{\svgwidth}
  \fi
  \global\let\svgwidth\undefined
  \makeatother
  \begin{picture}(1,0.75873059)%
    \put(0,0){\includegraphics[width=\unitlength]{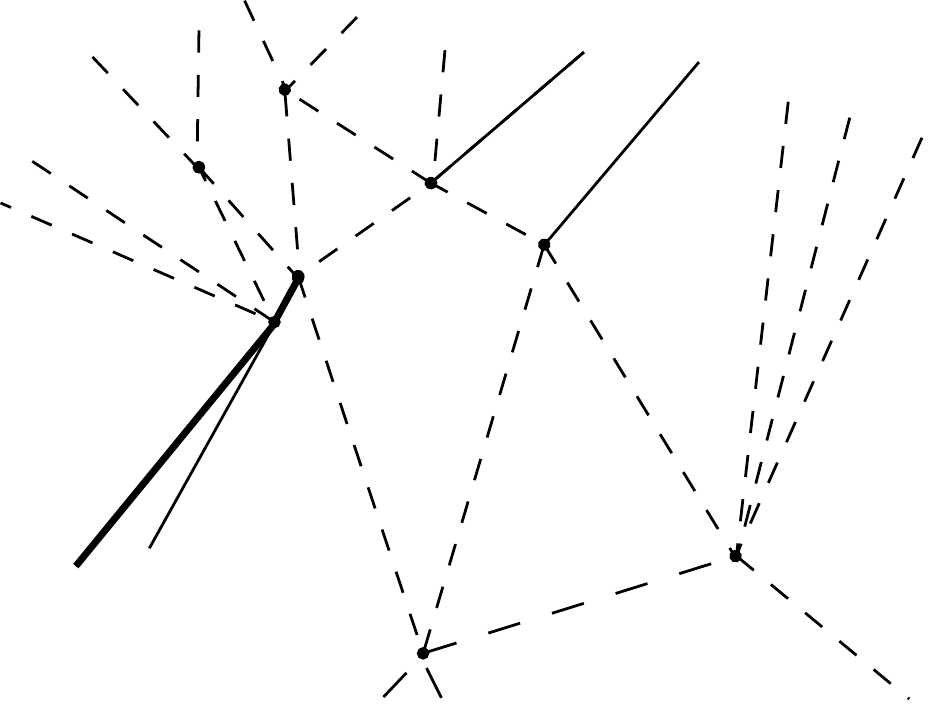}}%
    \put(0.08493751,0.21800345){\color[rgb]{0,0,0}\makebox(0,0)[lb]{\smash{$\gamma_1$}}}%
    \put(0.38,0.54){\color[rgb]{0,0,0}\makebox(0,0)[lb]{\smash{$\sigma_1$}}}%
    \put(0.33395796,0.44968022){\color[rgb]{0,0,0}\makebox(0,0)[lb]{\smash{$(t_1,x_1)$}}}%
  \end{picture}%
\endgroup